\documentclass[a4paper]{amsart}

\usepackage{amssymb}
\usepackage[colorlinks=true, linkcolor=blue]{hyperref}
\usepackage[all]{xypic}
\newif\ifFIRST\newdimen\MAXright\MAXright0pt
\def\dynkin{\mbox\bgroup\scriptsize
\let\SMALL\relax
\FIRSTtrue\hskip.5em%
\setbox1\hbox{$\diagup$}%
\setbox2\hbox{$\diagdown$}%
\setbox0\hbox to2\wd1{\hrulefill}%
\setbox8\hbox to.5\wd1{\hrulefill}%
\setbox3\hbox{$\bullet$}%
\setbox4\hbox{$\times$}%
\def\root##1{\ifFIRST\setbox5\hbox{$##1$}\ifdim\wd5>1.3em%
\hskip-.5em\hskip.5\wd5\fi\fi\FIRSTfalse%
\hskip-.25em\raise1.5\wd3\hbox to0pt{\hss\hskip.45em$%
##1$\hss}\copy3\hskip-.25em\setbox6\hbox{$##1$}%
\MAXright\wd6}%
\def\droot##1{\ifFIRST\setbox5\hbox{$##1$}\ifdim\wd5>1.3em%
\hskip-.5em\hskip.5\wd5\fi\fi\FIRSTfalse%
\hskip-.25em\lower1.8\wd3\hbox to0pt{\hss\hskip.45em$%
##1$\hss}\copy3\hskip-.25em\setbox6\hbox{$##1$}%
\MAXright\wd6}%
\def\rroot##1{\hskip-.25em\copy3\hbox to0pt{\hskip.3em$##1$\hss}%
\hskip-.25em\setbox6\hbox{\hskip.6em$##1##1$}%
\MAXright\wd6}%
\def\norroot##1{\hskip-.36em\copy4\hbox to0pt{\hskip.3em$##1$\hss}%
\hskip-.48em\setbox6\hbox{\hskip.6em$##1##1$}%
\MAXright\wd6}%
\def\noroot##1{\ifFIRST\setbox5\hbox{$##1$}\ifdim\wd5>1.3em%
\hskip-.5em\hskip.5\wd5\fi\fi\FIRSTfalse%
\hskip-.36em\raise1.5\wd3\hbox to0pt{\hss\hskip.6em$%
##1$\hss}\copy4\hskip-.38em\setbox6\hbox{$##1$}%
\MAXright\wd6}%
\def\nodroot##1{\ifFIRST\setbox5\hbox{$##1$}\ifdim\wd5>1.3em%
\hskip-.5em\hskip.5\wd5\fi\fi\FIRSTfalse%
\hskip-.36em\lower1.8\wd3\hbox to0pt{\hss\hskip.6em$%
##1$\hss}\copy4\hskip-.38em\setbox6\hbox{$##1$}%
\MAXright\wd6}%
\def\link{\raise.25em\copy0}%
\def\minilink{\raise.25em\copy8}%
\def\llink##1{\raise.35em\copy0\hskip-\wd0%
\raise.12em\copy0\hskip-.5\wd0\hbox to0pt{\hss$##1$\hss}\hskip.5\wd0}%
\def\lllink##1{\raise.40em\copy0\hskip-\wd0\raise.25em\copy0\hskip-\wd0%
\raise.10em\copy0\hskip-.5\wd0\hbox to0pt{\hss$##1$\hss}\hskip.5\wd0}%
\def\rootupright##1{\hbox to0pt{\raise.45em\copy1\hskip-.25em\raise1.3\ht1%
\hbox{\copy3\hskip.3em$##1$}\hss}%
\setbox6\hbox{\hskip.6em\copy1\copy1$##1##1$}%
\ifdim\MAXright<\wd6\MAXright\wd6\fi}%
\def\norootupright##1{\hbox to0pt{\raise.45em\copy1\hskip-.25em\raise1.3\ht1%
\hbox{\kern-.1em\copy4\hskip.3em$##1$}\hss}%
\setbox6\hbox{\hskip.6em\copy1\copy1$##1##1$}%
\ifdim\MAXright<\wd6\MAXright\wd6\fi}%
\def\rootdownright##1{\hbox to0pt{\raise-.5em\copy2\hskip-.25em\raise-1.35\ht1%
\hbox{\copy3\hskip.3em$##1$}\hss}\setbox6%
\hbox{\hskip.6em\copy2\copy2$##1##1$}%
\ifdim\MAXright<\wd6\MAXright\wd6\fi}%
\def\norootdownright##1{\hbox to0pt{\raise-.5em\copy2\hskip-.25em\raise-1.35\ht1%
\hbox{\kern-.1em\copy4\hskip.3em$##1$}\hss}\setbox6%
\hbox{\hskip.6em\copy2\copy2$##1##1$}%
\ifdim\MAXright<\wd6\MAXright\wd6\fi}%
\def\rootdown##1{\hbox to0pt{\hskip-.05em\vrule height.25em depth.65em%
\hskip-.25em\raise-.95em\hbox{\copy3\hskip.3em$##1$}\hss}%
\setbox6\hbox{$##1$}%
\ifdim\MAXright<\wd6\MAXright\wd6\fi}%
\def\norootdown##1{\hbox to0pt{\hskip-.05em\vrule height.25em depth.65em%
\hskip-.25em\raise-.95em\hbox{\copy4\hskip.3em$##1$}\hss}%
\setbox6\hbox{$##1$}%
\ifdim\MAXright<\wd6\MAXright\wd6\fi}%
\def\dots{\raise.25em\copy8\hskip.25em\raisebox{-0.025em}{$\cdots$}\hskip.10em%
\raise.25em\copy8}}%
\def\enddynkin{\ifdim\MAXright>1em\hskip.5\MAXright\else\hskip.5em\fi\egroup}%
 
This macro is used for creation of Dynkin diagrams in AMSTeX.
The usage:
In math-mode, between \dynkin and \enddynkin,
\root#1 makes a node with label #1 over this node,
\rroot#1 and \droot do the same with labels on right and down,
\noroot and \norroot do the same with cross instead of bullet,
\link links two adjacent \root's, \llink#1 links two adjacent nodes
with double-line with #1 (probably '>' or '<') in the middle, \lllink is does
same with three lines, \dots does \cdots between the two adjacent nodes,
\rootdown#1 makes a linked node down labeled by #1, \rootupright#1 and
\rootdownright#1 are labeled linked nodes in the indicated directions 
which do not count for the horizontal dimensions (so the user has to adjust
ad hoc the break after the diagram).

Examples:
$\dynkin \root{a_1}\link\root{a_2}\dots\root{a_{n-1}}\link\root{a_n}
\enddynkin$
or
$\dynkin \root{}\lllink>\root{}\enddynkin$
or
$\dynkin \root{a}\link\root{b}\rootupright{c}\rootdownright{d}\enddynkin$
or
$\dynkin \root{}\link\root{}\rootupright{}\link\root{}\enddynkin$.

\let\dyn\dynkin\let\edyn\enddynkin
\newcommand{\wgt}[2]{\hbox to2mm{\hss\scriptsize${(#1|#2)}$\hss}}
\newcommand{\wg}[1]{\hbox to2mm{\hss\scriptsize${(#1|)}$\hss}}

\swapnumbers

\makeatletter
\def\@secnumfont{\bfseries}
\makeatletter

\theoremstyle{thm}
\newtheorem{thm}[subsection]{Theorem}
\newtheorem{prop}[subsection]{Proposition}
\newtheorem*{prop*}{Proposition}
\newtheorem*{thm*}{Theorem}

\newtheorem*{lem*}{Lemma}

\newtheorem*{kor*}{Corollary}

\newtheorem*{definition*}{Definition}

\theoremstyle{definition} 

\def\frak{\mathfrak}
\def\Bbb{\mathbb}

\newcommand{\x}{\times}
\renewcommand{\o}{\circ}

\let\ccdot\cdot
\def\cdot{\hbox to 2.5pt{\hss$\ccdot$\hss}}

\newcommand{\al}{\alpha}

\newcommand{\la}{\lambda}
\newcommand{\om}{\omega}
\renewcommand{\phi}{\varphi}
\newcommand{\ph}{\varphi}
\newcommand{\si}{\sigma}

\newcommand{\ze}{\zeta}

\newcommand{\Ga}{\Gamma}
\newcommand{\La}{\Lambda}
\newcommand{\Ph}{\Phi}

\newcommand{\Om}{\Omega}

\begin{document}
\title[Semiholonomic jets and induced modules]{Semiholonomic jets and
induced modules in Cartan geometry calculus} 
\author{Jan Slov\'ak and Vladim\'{\i}r Sou\v cek}

\address{J.S.: Department of Mathematics and Statistics, Faculty of Science
Masaryk University, Kotl\'a\v rsk\'a 2a, 611~37~Brno, Czech Republic
\newline\indent V.S.: Mathematical Institute of Charles University,
Sokolovská 83, 186 75 Praha, Czech Republic}

\email{slovak@math.muni.cz, soucek@karlin.mff.cuni.cz}

\begin{abstract} 
The famous Erlangen Programme was coined by Felix Klein in 1872 as an
algebraic approach allowing to incorporate fixed symmetry groups as the core
ingredient for geometric analysis, seeing the chosen symmetries as intrinsic
invariance of all objects and tools.  This idea was broadened essentially by
Elie Cartan in the beginning of the last century, and we may consider
(curved) geometries as modelled over certain (flat) Klein's models.

The aim of this short survey is to explain carefully the basic concepts and
algebraic tools built over several recent decades.  We focus on the direct
link between the jets of sections of homogeneous bundles and the associated
induced modules, allowing us to understand the overall structure of
invariant linear differential operators in purely algebraic terms.  This
allows us to extend essential parts of the concepts and procedures to the
curved cases.\end{abstract}

\thanks{Both authors
acknowledge the support by the grants GX19-28628X and 
GA24-10887S of GA\v CR.
The article is also based on lectures by the first author at the 
Training school on Cartan geometry in Brno, September 2023, an event of 
the COST Action CaLISTA CA21109
supported by COST (European Cooperation in Science and Technology), further
support of the project CaLIGOLA, MSCA Horizon, project id 101086123, is
acknowledged by the first author, too.}


\maketitle

These notes go back to much earlier collaborative works of the authors, in
particular with A. Cap and M.G. Eastwood, cf. \cite{ES, CSS4}. The reader may
also see it as an extension of the recent notes \cite{SlSu} focusing on the
tractor calculi and BGG machinery from a quite different perspective. 

Linearized physical theories can be often viewed as complexes of
linear differential operators (and the laws of Physics are then modeled as
the equality of kernels and ranges of such operators), cf. \cite{CD,
E-notices} and the references therein. The expected
symmetries of the theory enforce the operators to commute with them,
thus, such operators have to be \emph{invariant}.

Our aim is to explain concepts allowing to discuss invariant linear
differential operators effectively.  In the homogeneous case, this will
become a very algebraic story, which we then (partially) extend to the 
curved geometries. The main ideas for that can be traced back to
\cite{ER,Ba1,ES}. In large extent, we adopt the language and notation from
\cite{ES}.
We believe that the reader will enjoy the power of the Cartan connections on
our journey.

If necessary, more background on functorial geometric constructions, Klein
geometries, and Cartan geometries can be found in \cite{KMS, parabook, Sha},
while the representation theory can be checked with \cite{vogan}.
We shall work in the category of smooth finite dimensional manifolds here.

\section{The algebraic story of the Klein geometries}

\subsection{Klein geometries and homogenous
bundles}\label{homogeneous-bundles}
A Klein geometry is a manifold $M$ with a transitive smooth action of a Lie
group $G$. Choosing a point $O\in M$, there is the isotropy subgroup $H$ of
this point and the identification $M=G/H$. Up to a choice of the origin $O$,
all Klein geometries are such homogeneous spaces $G/H$. At the infinitesimal
level, the quotient of the Lie algebras $\mathfrak g/\mathfrak h$ clearly
is naturally identified with the tangent space $T_OM$ at the origin. 

Notice that $G \to G/H$ is a principal $H$-bundle, and $G$ comes equipped
with the Maurer-Cartan form $\om\in\Om^1(G,\mathfrak g)$, the prototype of
Cartan connections.
 
Next, consider any linear representation $\mathbb{E}$ of $H$ and 
the associated bundle $\mathcal{E} = G \times_H \mathbb{E}$, i.e., the
classes of the equivalence relations on $G\times \mathbb E$ 
given by $ (u,v) \sim (u \cdot h, h^{-1} \cdot v)$ for all $h\in H$. The
tangent and cotangent bundles $TM$ and $T^*M$ are nice examples with the
$H$-modules $\mathfrak g/\mathfrak h$ and $(\mathfrak g/\mathfrak h)^*$, and
notice how the Maurer-Cartan form provides the identifications.

In the special case when $\mathbb E$ happens to be a $G$-module (and we consider
the restriction of the action to $H\subset G$), 
we may identify the class represented by 
$(u,v)$ with the couple $(u \cdot H, u \cdot v)$. 
Indeed, taking another representative leads to 
$$
((u\cdot h)\cdot H, (u\cdot h)\cdot (h^{-1}\cdot
v))=(u\cdot H,u\cdot v).
$$
Thus, we have verified that $\mathcal{E}$ is the trivial bundle 
$\mathcal{E} = M \times \mathbb{E}$  over $M=G/H$ for all $G$-modules
$\mathbb E$. 

On the other hand, \emph{homogeneous (vector) bundles} over a Klein geometry $M$ 
are the bundles with well defined $G$-actions by (vector) bundle morphisms.
Clearly, for each such bundle, the restriction of the action to the isotropy
group and the fiber over $O$ provides the $H$-module $\mathbb E$ and the
original homogeneous bundle is then identified with $\mathcal E = G\x_H
\mathbb E$. Moreover,
$H$-module morphisms lead to vector bundle morphisms between the homogenous
bundles in the obvious way.  

In other words, for each Klein geometry $M=G/H$, 
we have constructed a functor from the category of
$H$-modules to the category of homogenous bundles over $M=G/H$ with the obvious
action on morphisms.

Extending $G\to G/H$ to the principle $G$-bundle $\Tilde{G} = G \times_H
G\to G/H$, the Maurer-Cartan form $\omega$ uniquely extends to a principal
connection form $\Tilde \omega$ on $\Tilde G$.  

Finally, for $G$-modules $\mathbb T$ we can further identify $\mathcal{T}$
as the associated space $\mathcal{T} = \Tilde{G} \times_G \mathbb{T}\simeq M\x \mathbb T$
and we see that there is the induced linear connection $\nabla$ on all such
bundles $\mathcal{T}$.  These very special homogenous bundles are called
\emph{tractor bundles}.\footnote{These special bundles were traced back to
Tracy Thomas, who introduced them when searching for generalizations of tensor
bundles suitable for conformal Riemannian geometry, see \cite{BEG}. In
private communication with M.G. Eastwood, the authors of this note heard that the name
\emph{tractor} illustrates the fact that traction comes after tension, and
also the similarity with the name of the first inventor.}

\subsection{Sections of homogenous bundles and jet bundles}\label{1.2} 
A global version
of writing sections $\si\in\Ga(\mathcal E)$ of homogeneous bundles in
coordinates views the sections as functions $\tilde\si:G\to \mathbb E$ (by
abuse of notation, later we shall use the same letter $\si$ for both), which have to be $H$-equivariant, i.e., $\tilde\si(u\cdot
h)=h^{-1}\cdot \tilde\si(u)$.  Indeed, such a function defines the section
$\si$ with its values $\si(u\cdot H)$ represented by $(u,\tilde\si(u))$. 
Obviously, this is a well defined bijection between $\Ga(\mathcal E)$ and
$C^\infty(G,\mathbb E)^H$.

The (left) $G$-action $\ell_g$ on the homogeneous bundles induces, of
course, the action on the sections: $(g\cdot \si)(u\cdot H) =
\ell_g\circ\si\circ \ell_{g^{-1}}(u\cdot H)$, which means that in the other
picture, the action is $g\cdot \si = \si\circ\ell_{g^{-1}}$, which again
produces $H$-equivariant functions on $G$.

Next, let us have a look at the jet-prolongations $J^k\mathcal E$ of our
homogenous bundles $\mathcal E= G\x_H\mathbb E$. The $G$-action on sections
projects to the $G$-action on $J^k\mathcal E$, so that they are again
homogenous bundles and let us call the standard fiber $J^k\mathbb
E=(J^k\mathcal E)_O$ over the origin $O$ the \emph{$k$-jet prolongation} of the 
$H$-module $\mathbb E$. 

There is the straightforward observation:

\begin{prop}\label{prop.1.3}
The invariant linear differential operators $D:\Ga(\mathcal E)\to
\Ga(\mathcal F)$, of order at most $k$, are in
bijective correspondence with the $H$-module homomorphisms $J^k\mathbb
E\to \mathbb F$.
\end{prop}
\begin{proof}
Clearly, evaluating the values of an invariant differential operator $D$ at
the origin $O$, $D(\si)(O)$ depends on the $k$-jet $j^k_O\si$ only. By
restricting the invariance to $H\subset G$, we obtain the requested module
homomorphism by the linearity of $D$.

Vice versa, linear differential operators of order at most $k$ coincide with
morphisms between the homogeneous vector bundles $J^k\mathcal E$ and $\mathcal
F$, and those are in bijection with the module homomorphisms.   
\end{proof}

Although the latter observation looks promising, $J^k\mathbb E$ is a
horrible representation of $H$, even if $\mathbb E$ was nice, e.g.,
irreducible. Thus, in general, we can
hardly find and discuss the operators easily this way. Exceptionally, the case
$k=1$ might be discussed directly for large classes of Klein geometries, cf.
\cite{SlSo}.

\subsection{Induced modules}
A better way to understand invariant linear differential operators 
was suggested very long ago, see e.g., \cite{Ki,Ko}, and the references
therein. The point is that understanding embeddings of
nice modules into complicated ones might be much easier than looking for
morphisms in the original direction. Thus we look at the dual picture.

The elements $X$ of the Lie algebra $\mathfrak g$ are identified with the
left invariant vector fields $\om^{-1}(X)\in\mathcal X(G)$.  Differentiating
the $H$-equivariant functions $\si:G\to \mathbb E$ in the direction of
$X\in\mathfrak g$ in the unit $e\in G$ corresponds to derivatives of the
sections.  More precisely, if $X\in\mathfrak h$, then $(X\cdot\si)(u) =
-X\cdot (\si(u))$ by the equivariance and, thus, the genuine differential
parts are in the quotient $\mathfrak g/\mathfrak h$, thus corresponding to
derivatives of the sections in directions in $T_OM$.

Now, consider a ``word'' $X_1X_2\dots X_k$ of elements in $\mathfrak g$ and
the corresponding differential operator $\si\mapsto
\om^{-1}(X_1)\o\om^{-1}(X_2)\o\dots\o\om^{-1}(X_k)\cdot\si(e)$ on the functions. 

We may consider this operation as defined on the tensor algebra $T(\mathfrak
g)$ and obviously the entire ideal in $T(\mathfrak g)$ generated by the
expressions $X\otimes Y-Y\otimes X - [X,Y]$, with $X, Y\in \mathfrak g$ and
$[X,Y]$ their Lie bracket, must act trivially.

The resulting quotient (left and right) $\mathfrak g$-module $\mathfrak U(\mathfrak g)
= T(\mathfrak g)/\langle X\otimes Y - Y\otimes X - [X,Y]\rangle$ is called
the \emph{universal enveloping algebra} of the Lie algebra $\mathfrak g$. 

We would like to understand the linear forms on the jet modules
$J^k\mathbb E$.  So far we differentiate functions also in the vertical
directions, and our values are in $\mathbb E$.  Thus we should consider the
tensor product 
$$
V(\mathbb E)=\mathfrak U(\mathfrak g)\otimes_{\mathfrak
U(\mathfrak h)} \mathbb E^*
.$$ 
The space $V(\mathbb E)$ clearly enjoys the
structure of a $(\mathfrak g,H)$-module (and $(\mathfrak U(\mathfrak
g),H)$-module), and it is called the \emph{induced module} for the
$H$-module $\mathbb E$.

\begin{prop}\label{prop.1.5}
The induced module $V(\mathbb E)$ is the space of all linear forms on
$J^\infty\mathbb E$ which factor through some $J^k\mathbb E$, i.e., depend
on finite number of derivatives. 
\end{prop}
\begin{proof}
The claim follows from the construction of $V(\mathbb E)$ and the fact that
choosing a complementary vector subspace to $\mathfrak h$ in $\mathfrak g$,
we can decompose all letters in our words $X_1\dots X_k$ above and, by the
equalities enforced by living in the quotient by the ideal, we may ``bubble'' 
the letters in $\mathfrak h$ to the very right. Once there, they act 
algebraically and, thus,
tensorizing over $\mathfrak U(\mathfrak h)$ we remove just all redundancies.

The reader might consult \cite{vogan} for more details.
\end{proof}

Obviously, $\mathbb E^*$ injects into $V(\mathbb E)$ and generates this
$\mathfrak g$-module. Now, we may enjoy a small but extremely important miracle:

\begin{thm}[Frobenius reciprocity]\label{thm.1.6}
For all finite dimensional representations $\mathbb E$ and $\mathbb F$ of $H$, 
there are the canonical isomorphisms 
$$
\operatorname{Hom}_H(\mathbb F^*,V(\mathbb E)) =
\operatorname{Hom}_{(\mathfrak U(\mathfrak g),H)}(V(\mathbb F),V(\mathbb E)).
$$
\end{thm}

\begin{proof}
If we are given a homomorphism $\Ph\in 
\operatorname{Hom}_{(\mathfrak U(\mathfrak g),H)}(V(\mathbb F),V(\mathbb E))$,
we simply define $\ph:\mathbb F^*\to V(\mathbb E)$ by restriction.

On the other hand, having a $\ph\in 
\operatorname{Hom}_H(\mathbb F^*,V(\mathbb E))$, we first define for all
$x\in\mathfrak U(\mathfrak g)$ and $v\in \mathbb F^*$,
$$
\Phi(x\otimes v) = x \otimes_{\mathfrak U(\mathfrak h)} \ph(v)
,$$
which extends linearly, if well defined. To check this, notice that for all
$X\in\mathfrak h$ and $v\in\mathbb F^*$,
$$
\Ph(X\otimes v - 1\otimes X\cdot v) = X\otimes \ph(v) - 1\otimes \ph(X\cdot
v) = X\otimes \ph(v) - 1\otimes X\cdot\ph(v)
,$$ which completes the proof.
\end{proof}

\section{The translation principle}

\subsection{Parabolic Klein models} In the rest of the paper, we shall restrict
to a large class of geometries modelled over the Klein geometries $G/P$ with
$G$ semisimple and $P\subset G$ parabolic.  
Let us recall that $P$ is called
parabolic if it contains a Borel subgroup in $G$.

We expect the reader
knows the elements of the structure theory of the semisimple Lie groups, the
root spaces, the Weyl group, etc. Consult \cite{vogan} or \cite[Chapter 2]{parabook} if necessary.

At the level of Lie algebras, the parabolic subalgebras are 
those which contain a Borel subalgebra.   
The choices of all parabolic subalgebras $\mathfrak p\subset \mathfrak g$
correspond to graded decompositions of the semisimple Lie algebras 
\begin{equation}\label{adjoint_composition}
\mathfrak{g} = \mathfrak{g}_{-k} \oplus \dots \oplus \mathfrak
g_{-1}\oplus\mathfrak g_{0}\oplus\mathfrak g_1\oplus\dots\oplus\mathfrak{g}_k = \mathfrak
g_-\oplus \mathfrak p.
\end{equation}
Thus, the Lie brackets satisfy $[\mathfrak g_i,\mathfrak
g_j]\subset \mathfrak g_{i+j}$, and $\mathfrak p=\mathfrak g_0 \oplus
\mathfrak p_+=\mathfrak g_0 \oplus \mathfrak g_1\oplus\dots\oplus \mathfrak g_k$
is the decomposition of the parabolic subalgebra $\mathfrak p$ 
into the reductive Levi quotient 
$\mathfrak l=\mathfrak g_0$ and
the nilradical $\mathfrak p_+$. There is also the subalgebra $\mathfrak
g_{-}=\mathfrak g_{-k}\oplus\dots\dots\oplus \mathfrak g_{-1}$ 
complementary to $\mathfrak p$, which is the dual to $\mathfrak p_+$
with respect to the Cartan-Killing form on $\mathfrak g$. 

For a given semisimple $\mathfrak g$, the above gradings with isomorphic
parabolic subalgebras $\mathfrak p\subset \mathfrak g$ are given uniquely, 
up to conjugation in $\mathfrak g$. They are also uniquely determined by the 
\emph{grading elements} $E$, i.e., $E$ with the property
$[E,X]=jX$ for all $X\in\mathfrak g_j$. Obviously, $E$ is in the center of
$\mathfrak g_0$, which is identified with 
$\mathfrak l=\mathfrak p/\mathfrak p_+$. 

The closed Lie subgroups $P\subset G$ are parabolic if and only if their
algebras $\mathfrak p=\operatorname{Lie}P$ are parabolic.

If $G$ is a complex semisimple Lie group, then there is a nice geometric
description: $P\subset G$ is parabolic if and
only if $G/P$ is a compact manifold (and then it is a compact K\"ahler
projective variety),
see e.g., \cite[Section 1.2]{Zierau}. In
the real setting, the so called generalized flag varieties $G/P$ with
parabolic $P$ are always compact, too.

We talk about $|k|$-graded Klein models $G/P$. For the complex semisimple
algebras, the Borel subalgebras are generated by Cartan subalgebras in
$\mathfrak g$ and all simple positive co-roots $\al_i$ in $\mathfrak g$. The
parabolic subalgebras $\mathfrak p$ then correspond to the subsets of the
co-roots $\al_i$, for which $-\al_i$ do not belong to $\mathfrak p$. In the language
of the Dynkin diagrams, this can be nicely encoded by crossing the nodes
related to negative simple co-roots in $\mathfrak g_-$.

In the real situation, $\mathfrak p\subset\mathfrak g$ is parabolic, if the
same holds true for the complexification. The classification is more subtle
here, but it can be nicely encoded by the
Satake diagrams with crossed nodes being allowed only for the white ones,
and if one of the nodes joined by an arrow is crossed, then the other one
has to be crossed, too (see \cite[Chapter 2]{parabook} for detailed discussion). 

The induced modules $V(\mathbb E)$ are called (generalized) \emph{Verma
modules} and they enjoy a very rich and well understood structure theory, see e.g.,
\cite{Lep, BC1, BC2}. 

We shall present a brief selection of tools and
results from this theory, 
preparing our approach to invariant differential operators on curved
geometries.

A panopticum of examples of geometries modelled on parabolic Klein geometries
can be found in \cite[Chapter 4]{parabook}, including projective, conformal
Riemannian, CR, and many others. Similarly to the survey \cite{SlSu}, we
shall focus on a few $|1|$-graded examples here. 

We shall see that the invariant linear operators appear in isomorphic
patterns and the de Rham complex of differential operators on the algebra of
differential forms (decomposed into irreducible components) is a prototype
of all of them.  We shall try to explain the general approach in a "learning
by doing" way and we go through
one line of examples of the geometries and patterns only.

\subsection{Grassmannian examples} \label{Grass_ex}
Let us work with the real split forms of
the type $A$ algebras, i.e., $\mathfrak g =\mathfrak{sl}(n,\mathbb R)$. 
There the $|1|$-graded cases correspond to choices of just one crossed node,
say the $p$th one, and $\mathfrak g_{-1}=\mathbb R^q\otimes(\mathbb R^p)^*$,
with $p+q=n$.  The Klein geometries are the so-called $(p,q)$-Grassmannians,
i.e., the spaces of $p$-planes through origin in $\mathbb R^{p+q}$.  The
structure of the graded $\mathfrak g$ is nicely seen in the blockwise scheme
of all trace-free matrices $X\in\mathfrak g$: $$ \mathfrak g \simeq
\begin{pmatrix} 0 & 0 \\ \mathfrak g_{-1} & 0 \end{pmatrix} \oplus \mathfrak
z \oplus \begin{pmatrix} \mathfrak{sl}(p,\mathbb R) & 0 \\ 0 &
\mathfrak{sl}(q,\mathbb R) \end{pmatrix} \oplus \begin{pmatrix} 0 &
\mathfrak g_1 \\ 0 & 0 \end{pmatrix} ,$$ where $\mathfrak z$ is the
one-dimensional center of the Levi factor $\mathfrak g_0$ (generated by the
grading element $E$), while the rest of $\mathfrak g_0$ is its semisimple
part.  Obviously, the grading element is built of constant multiples of
identity matrices in the diagonal blocks, adjusted depending on $p$ and $q$.

The case $p=1$ provides the projective spaces $\mathbb{RP}_{n-1}$. We shall
always assume $1\le p\le q$.
In the small dimensions and in the language of the Dynkin diagrams, the
2-dimensional projective space is drawn as 
$\dyn\noroot{}\link\root{}\edyn$, while our 
Klein geometries of, 
$(2,2)$, $(2,3)$, and $(3,3)$ Grassmannians are encoded as
$$
\dyn\root{}\link\noroot{}\link\root{}\edyn \qquad
\dyn\root{}\link\noroot{}\link\root{}\link\root{}\edyn \qquad
\dyn\root{}\link\root{}\link\noroot{}\link\root{}\link\root{}\edyn
$$
Notice, that $\mathfrak{sl}(4,\mathbb C)=\mathfrak{so}(6,\mathbb C)$ and
this explains how the case of $(2,2)$-Grassmann\-ians corresponds to the
Roger Penrose's complexified model of the Universe (i.e., the four-dimensional
conformal Riemannian geometry of the Minkowski space, but in the
complexified form).

As well known, all representations of a semisimple Lie algebra are
completely reducible and the irreducible ones can be all built from the so
called fundamental weights, which form a base of the dual of the Cartan
subalgebra in $\mathfrak g$, and we can encode them by writing one over the
respective node in the Dynkin diagram and zeros on the rest. 
All irreducible representations are then given
by the linear combinations of the fundamental ones with non-negative
integral coefficients. These so-called \emph{dominant weights} provide us
with the (homogeneous) irreducible tractor bundles, i.e., 
those homogeneous vector bundles defined by irreducible $G$-modules.

\subsection{The completely reducible homogenous
bundles}\label{index_notation}
Let us continue with our Grassmannian example.  All irreducible $\mathfrak
p$-modules are obtained as (outer) tensor products of irreducible
representations of the two semisimple components in $\mathfrak g_0$,
together with the action of the center, and the trivial action of $\mathfrak
g_1$.  The simplest nontrivial modules include those with trivial actions of
the center and the second $\mathfrak g_0$ component, thus, defined by
irreducible actions of $\mathfrak{sl}(p,\mathbb R)$.  The standard
representation $\mathbb R^p$ leads to the homogenous bundle $\mathcal
E^{A}$, while its dual defines the homogenous bundle $\mathcal E_{A}$. 
Similarly, if only the $\mathfrak{sl}(q,\Bbb R)$ component acts
nontrivially, the standard representation will give rise to the homogenous
bundle $\mathcal E^{A'}$, and its dual will be $\mathcal E_{A'}$.  The tangent
bundle $\mathcal E^{A'}_{A}$ and cotangent bundle $\mathcal E^{A}_{A'}$ come
from the tensor products of these bundles.  Notice, this is exactly the
Penrose's abstract index notation with the spinor indices, extended to
general Grassmannians.  Finally, the one-dimensional modules
$\Lambda^p(\mathbb R^p)\simeq \Lambda^q(\mathbb R^{q*})$ are coming from
actions of the center $\mathfrak z$ only.  We call such bundles the
\emph{weight bundles $\mathcal E[w]$}, and we adopt the usual normalization
taking \begin{gather*} \mathcal E^{[AB\dots C]}=\mathcal E[-1]=\mathcal
E_{[A'B'\dots C']} \\ \mathcal E_{[AB\dots C]}=\mathcal E[1]=\mathcal
E^{[A'B'\dots C']} ,\end{gather*} where we put $p$ indices on the left and $q$
indices on the right, and $[\ \ ]$ or $(\ \ )$ on indices mean
antisymmetrization or symmetrization, respectively.  The general weight
bundles with integral weights are the tensor powers $\mathcal E[w]=\mathcal
E[1]^w$ for positive $w$ and $\mathcal E[w]=\mathcal E[-1]^{-w}$ for
negative ones.  With some care, we may extend this to real weights $w$.

It is easy to encode the irreducible $\mathfrak p$ modules by weights of the
entire $\mathfrak g$. Indeed, the highest weights of such $\mathfrak
p$-modules are those integral linear combinations of the fundamental
weights of $\mathfrak g$ (recall the fundamental weights correspond to 
exterior forms $\Lambda^k\mathbb
R^{p+q}$, $k=1,\dots,p+q-1$), whose coefficients must by non-negative, 
except the
one over the crossed node. The coefficients over the uncrossed nodes
encode the representations of the two semisimple components in $\mathfrak
g_0$, while the coefficient over the crossed node completes the
information about the action of the center $\mathfrak z$. We call such weights
\emph{$\mathfrak p$-dominant}.

In our case, all our irreducible homogeneous vector bundles live in the
tensor bundles $\mathcal E^{A\dots BC'\dots D'}_{E\dots FG'\dots H'}\otimes\mathcal
E[w]$ which we usually write as 
$\mathcal E^{A\dots BC'\dots D'}_{E\dots FG'\dots H'}[w]$,
and there is a straightforward algorithm, how to compute the action of the
grading element on the corresponding modules.\footnote{In general, 
the action of the grading element of a
parabolic subalgebra $\mathfrak p$ in semisimple $\mathfrak g$ is computed
as the (sum of) scalar product(s) of the lines in the inverse Cartan matrix
of $\mathfrak g$ corresponding to the crossed nodes, with the coefficients of the highest weight (in the
expression via the fundamental weights), cf. \cite[Section 3.2.12]{parabook}.} 

First, we adopt another encoding for the weights. 
Writing $e_i$ for the diagonal matrix with just one 1 entry at the $i$th
place and zero otherwise, the dual basis to the fundamental weights
corresponds to $e_i-e_{i+1}$ and we may replace the $(n-1)$-tuples
$(\al_1,\dots,\al_{n-1})$ of
integers by $n$-tuples $(a_1,\dots,a_n)$, so that $\al_i=a_i-a_{i+1}$, and we
indicate the $p+q$ blockwise structure by adding a vertical bar at the proper place.
Requesting the dominant weights to be non-negative integral then means 
$a_1\ge a_2\ge \dots \ge a_n$, while for $\mathfrak p$-dominant weights we
request this for the first $p$-tuple and last $q$-tuple separately. Of
course, these longer vectors are unique for the weights, up to a common
constant only. Thus, we usually normalize the choice, e.g., we may request $a_n=0$.

Next, we write down the $\mathfrak p$-dominant weights corresponding 
to enough $\mathfrak p$-modules with known action of $E$ (e.g., $\mathfrak g$
and the fundamental representations the two semisimple components of 
$\mathfrak g_0$), and compute which linear formula provides the right action of $E$
on them. In our three examples displayed above, the grading elements
acts on the modules determined by the coefficients $a_i$ by the constants:
\begin{equation}\label{ell_ntuples}
\tfrac12(a_1+a_2-a_3-a_4)\quad \tfrac35(a_1+a_2)-\tfrac25(a_3+a_4+a_5)\quad
\tfrac12(a_1+a_2+a_3-a_4-a_5-a_6)
\end{equation}
clearly independent of our normalization.

Moreover, there is the so-called lowest weight $\rho$ being the sum of all the
fundamental ones (i.e., with coefficient one over each node). For good
reasons which we shall see below, we shall add $\rho$ to our weights when
encoding them. Thus the trivial representation will be written as 
$((n-1)\ (n-2)\ \dots\ (n-p)\mid (n-p-1)\ \dots\ 0)$. In particular, in
our three examples we get the three vectors
$$
(3\ 2\mid 1\ 0)\qquad (4\ 3\mid 2\ 1\ 0)\qquad (5\ 4\ 3\mid 2\ 1\ 0)
.$$
Using the same formulae for the action of the center will simply add the
constant $\frac12pq$ corresponding to the action on the lowest weight.

The $n$-tuples for $\mathfrak g$-dominant weights in this encoding satisfy
$a_1>a_2>\dots>a_{n-1}>a_n=0$. The $\mathfrak p$-dominant ones request only
$a_1>a_2>\dots>a_p$ and $a_{p+1}>\dots>a_n$.

\subsection{Homomorphisms between Verma modules}\label{2.4}
We continue working with semisimple $\mathfrak g$ and parabolic $\mathfrak p$.
Consider two $\mathfrak p$-modules $\mathbb E$, $\mathbb F$ and the
corresponding homogeneous bundles $\mathcal E$ and $\mathcal F$ (we shall
not care about the choice of the Lie groups now). 

As obvious from Propositions \ref{prop.1.3} and \ref{prop.1.5}, and
Theorem \ref{thm.1.6}, all invariant linear differential operators
$\Ga(\mathcal E)\to \Ga(\mathcal F)$ are uniquely defined by $(\mathfrak
U(\mathfrak g),P)$-homo\-morphisms $\Ph:V(\mathbb F)\to V(\mathbb E)$.

The Verma modules $V(\mathbb E)$ are always equipped by an obvious filtration 
$$
\mathbb R\subset(\mathbb R\oplus\mathbb E^*)=V_1(\mathbb E)\subset
V_2(\mathbb E)\subset \dots\subset V(\mathbb E).
$$  
Since $\mathbb F^*$ is finite dimensional, its image under $\Ph$ is
contained in $V_k(\mathbb E)$ for some integer $k>0$. The least $k$ with this property
is called the \emph{order} of the morphism $\Ph$. Clearly, this corresponds
to the differential order of the corresponding linear differential operator.

We may also translate the concept of the symbol into this dual setup.  
By virtue of the Poincar\'e-Birkhoff-Witt theorem, the grading corresponding
to the filtration is $\operatorname{gr}(V(\mathbb E))=\operatorname{gr}(\mathfrak
U(\mathfrak g_-))\otimes \mathbb E^* = S(\mathfrak g_-)\otimes_{\mathbb R} \mathbb
E^*$, (see e.g. \cite[Section 2.1.10]{parabook}). Thus $V(\mathbb E)=S(\mathfrak g_-)\otimes_{\mathbb R} \mathbb E$ as
a $\mathfrak g_0$-module. In particular, there are the short exact sequences 
\begin{equation}\label{sequence_for_symbol}
0 \to V_{k-1}(\mathbb E) \to V_{k}(\mathbb E)\to S^k\mathfrak g_-\otimes
\mathbb E \to 0
.\end{equation}

Now, the \emph{symbol} of a homomorphism $\Ph:V(\mathbb F)\to V(\mathbb E)$ 
is defined as 
$$
\si(\Ph) : \mathbb F^* \to V_k(\mathbb E)\to V_k(\mathbb E)/V_{k-1}(\mathbb E)= S^k(\mathfrak
g_-)\otimes \mathbb E^*
,$$
where $k$ is the order of $\Ph$.

The center of the universal enveloping algebra
$\mathfrak U(\mathfrak g)$ is quite well understood for all semisimple
algebras $\mathfrak g$. 
For each Verma module $V(\mathbb E)$, the
restriction of the action to this center is called the \emph{infinitesimal
character} of $V(\mathbb E)$. Clearly, the infinitesimal characters of
$V(\mathbb F)$ and $V(\mathbb E)$ must coincide if there should be a
non-zero homomorphism between them.

Now, another miracle comes: By the famous Harish-Chandra theorem,
\emph{two
Verma modules $V(\mathbb E)$ and $V(\mathbb F)$ have got the same infinitesimal
character if and only if their highest weights appear in the same orbit of
the affine action of the Weyl group on the space of weights}. 

Here, the Weyl group is generated by all reflections defined by the simple
roots, and the affine action is this very action applied to the sums of the
weights with the lowest weight $\rho$. 

In our Grassmannian examples, all the elements of the Weyl group
act just by the permutations of the coefficients in the $n$-tuple representing
the weight. Thus, for our special cases (we add the 2-dimensional projective
space), 
we are getting the following patterns of all $\mathfrak
p$-dominant weights in the affine orbit of the trivial representation.
Notice, we organize the columns by decreasing constant of the action of the
grading element, and the order of the homomorphisms between the modules in
the neighboring modules in the neighboring columns is always 1. 

In fact the columns are giving the
decompositions of $\La^j(T^*M)$ into irreducible components and the arrows
correspond to the restrictions and decompositions of the exterior 
differential $d$. Of course, as homomorphisms of the Verma modules, they go
in the opposite directions. 

\begin{equation}\label{1|2}
\xymatrix@R=2mm@C=4mm{
{\ \ \wgt{2}{10}\ \ }
&{\ \ \wgt{1}{20}\ \ }\ar[l]
&{\ \ \wgt{0}{21}\ \ }\ar[l]
}
\end{equation}

\begin{equation}\label{2|2}
\xymatrix@R=2mm@C=4mm{
\wgt{32}{10}
&&\wgt{21}{30} \ar[dl]
&&\wgt{10}{32} \ar[dl]
\\
&\wgt{31}{20} \ar[ul]
&&\wgt{20}{31} \ar[ul]\ar[dl]
\\
&&\wgt{30}{21} \ar[ul]
}
\end{equation}

\begin{equation}\label{2|3}
\xymatrix@R=2mm@C=4mm{
\wgt{43}{210}
&&\wgt{32}{410}\ar[dl]
&&\wgt{21}{430}\ar[dl]
&&\wgt{10}{432}\ar[dl]
\\
&\wgt{42}{310}\ar[ul]
&&\wgt{31}{420}\ar[ul]\ar[dl]
&&\wgt{20}{431}\ar[ul]\ar[dl]
\\
&&\wgt{41}{320}\ar[ul]
&&\wgt{30}{421}\ar[ul]\ar[dl]
\\
&&&\wgt{40}{321}\ar[ul]
}
\end{equation}

\begin{equation}\label{3|3}
\xymatrix@R=2mm@C=8mm{
&&&\wgt{432}{510}\ar[dl]
&\wgt{431}{520}\ar@(ul,dl)[l]\ar[ddl]
&\wgt{421}{530}\ar@(ul,dl)[l]\ar[ddl]
&\wgt{321}{540}\ar@(ul,dl)[l]
\\
&&\wgt{532}{410}\ar[dl]
&&&&&\wgt{320}{541}\ar[ul]\ar[dl]
\\
\wgt{543}{210}
&\wgt{542}{310}\ar@(ul,dl)[l]
&&\wgt{531}{420}\ar[dl]\ar[ul]
&\wgt{521}{430}\ar@(ul,dl)[l]
&\wgt{430}{521}\ar[uul]\ar[ddl]
&\wgt{420}{531}\ar@(ul,dl)[l]\ar[uul]\ar[ddl]
&&\wgt{310}{542}\ar[ul]\ar[dl]
&\wgt{210}{543}\ar@(ul,dl)[l]
\\
&&\wgt{541}{320}\ar[ul]
&&&&&\wgt{410}{532}\ar[ul]\ar[dl]
\\
&&&\wgt{540}{321}\ar[ul]
&\wgt{530}{421}\ar@(ul,dl)[l]\ar[uul]
&\wgt{520}{431}\ar@(ul,dl)[l]\ar[uul]
&\wgt{420}{531}\ar@(ul,dl)[l]
}
\end{equation}

There are algorithms discovering all non-zero homomorphisms in such
patterns, cf. \cite{BC1,BC2}. There are no other non-zero homomorphisms for
the de Rham complex on the 2-dimensional projective space in \eqref{1|2}.
In \eqref{2|2}, the central diamond (the
square of arrows) displays
two non-zero compositions, which equal each other up to sign (as expected,
since the whole pattern must be the de Rham complex of operators). Moreover,
there is the special homomorphism of fourth order joining the most right and
most left modules. This corresponds to the Paneitz operator, whose
symbol is the square of the Laplacian, see \cite{ES}. All other compositions
of morphisms not mentioned above are zero.

In \eqref{2|3} and \eqref{3|3}, the situation is the same with all the
diamonds there.  In \eqref{2|3}, there are additionally two fourth order
homomorphisms joining the modules in the first line.  In \eqref{3|3}, there
are six such fourth order morphisms, but also two quite different ones -- a
morphism of order nine joining the most right and most left modules, and one of
order seven joining their neighbors. We shall not go into details of this
example here, there is the work in progress, \cite{SlSo-prep}, covering this
case in detail.

Notice that the decompositions of the spaces of exterior forms
$\Lambda^k((\mathbb R^p)^*\otimes\mathbb R^q)$ are easily
understood by mimicking every symmetrization in $\otimes^k(\mathbb R^p)^*$
by identical antisymmetrization in $\otimes^k(\mathbb R^q)$, and vice versa.
This is very nicely encoded by means of the so called Young symmetrizers and
Young diagrams. The pattern \eqref{3|3} is drawn in this language in
\cite[Section 3.2.17]{parabook}.

Actually, another good way to understand the Grassmannians is to identify
the operators coming from the lower dimensional ones (corresponding to
forgetting some of the nodes in the Dynkin diagrams).  The following diagram
rewrites \eqref{3|3} this way, and completes all the long arrows (not seen
in the de Rham complex directly).  Viewing it as a 3-dimensional picture, we
can clearly see the parts corresponding to the $(2,3)$, $(2,2)$, and $(1,2)$
Grassmannians there.  Some of these are indicated by colors (we write only
the first half of the weight encodings, which determines the rest):

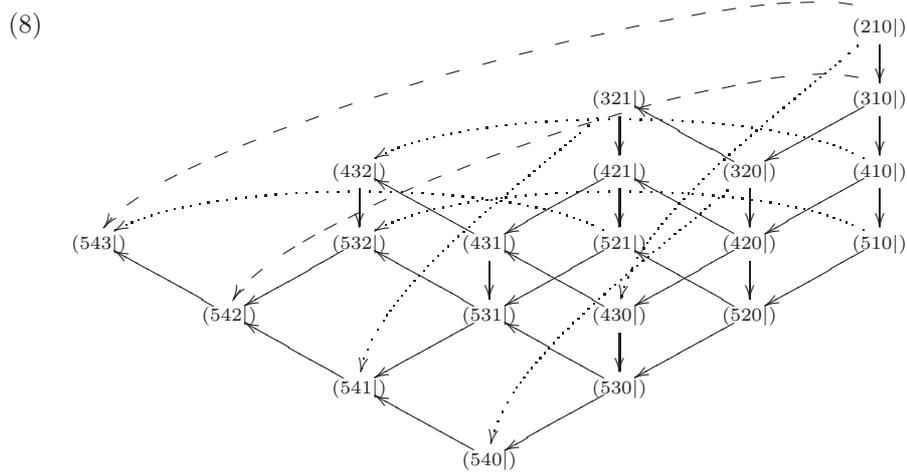
\begin{equation}\label{3|3-complete}
\xymatrix@R=5mm@C=13mm{
&&&&&& \wg{210} 
\ar@(u,u)@{-->}[llllllddd]
\ar[d]
\ar@(dl,u)@{..>}[ddddll]
\\
&&&& \wg{321} 
\ar[d] 
\ar@(dl,u)@{..>}[ddddll]
&& \wg{310}
\ar[d] \ar@[green][dl]
\ar@(u,u)@{-->}[lllllddd]
\\
&& \wg{432}
\ar[d]
&& \wg{421}
\ar@[blue][ld] \ar[d] 
& \wg{320}
\ar@[green][ul] \ar[d] \ar@(dl,u)@{..>}[ddddll]
&\wg{410}
\ar@(ul,ur)@{..>}[llll]
\ar@[blue][dl] \ar[d]
\\
\wg{543}
&& \wg{532}
\ar@[red][dl] 
& \wg{431}
\ar@[blue][ul] \ar[d] 
& \wg{521}
\ar@[red][dl] \ar@(ul,ur)@{..>}[llll]
& \wg{420}
\ar@[blue][ul] \ar[d] 
\ar@[blue][dl]  
& \wg{510}
\ar@[red][dl] \ar@(ul,ur)@{..>}[llll]
\\
& \wg{542}
\ar@[red][ul] 
&& \wg{531}
\ar@[red][ul] \ar@[red][dl] 
& \wg{430}
\ar[d] \ar@[blue][ul]
& \wg{520}
\ar@[red][ul] \ar@[red][dl]
\\
&& \wg{541}
\ar@[red][ul] 
&& \wg{530}
\ar@[red][ul] \ar@[red][dl]
\\
&&& \wg{540}
\ar@[red][ul] 
} 
\end{equation}

Notice, how the higher dimensional cases inherit the 
exceptional `long' operators, while adding some new ones, too. See
\cite{SlSo-prep} for more details about non-zero compositions of arrows in
\eqref{3|3-complete}. 

\subsection{Translation principle idea}
Let us come back to the general theory. 
The homomorphisms between Verma modules appear with striking regularity.
This fact is a consequence of another quite straightforward observation:

Suppose $\mathbb W$ is a $G$-module (and $P$-module by restriction), and
$\mathbb E$ a $P$-module. Then we
may view $\mathfrak U(g)\otimes \mathbb E^*\otimes \mathbb W^*$, as a
$\mathfrak g$-module, in two
different ways:
\begin{gather*}
X(x\otimes e \otimes w) = Xx\otimes e\otimes w
\\
X(x\otimes e \otimes w) = Xx\otimes e\otimes w + x\otimes e \otimes Xw
,\end{gather*}
for all $X\in\mathfrak g$, $x\in\mathfrak U(\mathfrak g)$, $e\in\mathbb E^*$,
and $w\in\mathbb W^*$.
The first one descends to the $\mathfrak g$-module (and also $(\mathfrak
U(\mathfrak g),P)$-module) structure on $V(\mathbb
E\otimes\mathbb W)$, while the second one yields the structure of $V(\mathbb
E)\otimes \mathbb W^*$.

Clearly, there is the unique $(\mathfrak U(\mathfrak g), P)$-module 
homomorphism $\ph$ between the above modules, defined as 
identity on $1\otimes e \otimes w$. An easy check reveals that $\ph$
descends to the isomorphism
\begin{equation}\label{twisting}
V(\mathbb E\otimes \mathbb W) = V(\mathbb E)\otimes \mathbb W^*
.\end{equation}

Next, an arbitrary non-trivial irreducible $G$-module 
$\mathbb W$ is never
an irreducible $P$-module. On the contrary, the $\mathfrak g_0$-orbit of the
action containing the highest weight vector of $\mathbb W$ 
forms the $\mathfrak p$-irreducible component $\mathbb W_\al$ 
on which the grading element acts by the biggest scalar $\al$, and 
the entire $\mathbb W$ enjoys a composition series
\begin{equation}\label{composition_series}
\mathbb W = \mathbb W_{\al-\ell} + \mathbb W_{\al-\ell+1} + \dots + \mathbb
W_{\al}
\end{equation} 
where the labeling reflects the scalar action by the grading element, and
the `right ends' $W_j=\mathbb W_j+\dots+\mathbb W_\al$ form $\mathfrak p$-submodules, i.e., we get the filtration
\begin{equation}\label{filtration}
\mathbb W_\al = W_\al\subset W_{\al-1}\subset\dots\subset
W_{\al-\ell}=\mathbb W,
\end{equation}
with $\mathbb W_j=W_j/W_{j+1}$. 
As a $\mathfrak g_0$-module,
the composition series is a direct sum of submodules $\mathbb W_j$ and 
each of them further decomposes into $\mathfrak g_0$-irreducible submodules 
$\mathbb W_{j,k}$. The composition series \eqref{adjoint_composition} of
the adjoint representation on $\mathfrak g$ is a good example.

Now, consider an irreducible $\mathfrak p$-module $\mathbb E$, and the
module $\mathbb
W$ as before. Then we arrive at the composition series
$$
\mathbb E\otimes \mathbb W = \mathbb E\otimes \mathbb W_{\al-\ell}+ \dots +
\mathbb E\otimes \mathbb W_\al
,$$
and each $\mathbb E\otimes \mathbb W_i$ splits into direct sum of
irreducible $\mathfrak g_0$-modules $\mathbb E_{i,j}$.  

Finally, assume that one of the many modules $\mathbb E'=\mathbb E_{i,j}$
has got a distinct infinitesimal character then all the other modules 
in the above decomposition. Then the injection $V(\mathbb E')\to
V(\mathbb E\otimes \mathbb W)$ is defined by its image being the joint
eigenspace of the infinitesimal character of $V(\mathbb E')$. Consequently,
there is the complementary subspace defined as the generalized eigenspaces
of all the other infinitesimal characters there. 

Thus, under the latter assumption, \emph{$V(\mathbb E')$ canonically splits off
the $V(\mathbb E\otimes \mathbb W)=V(\mathbb E)\otimes \mathbb W^*$ as a
direct summand.}

Now we are ready to tell the translation idea: If $\Ph:V(\mathbb F)\to V(\mathbb
E)$ is a non-trivial $(\mathfrak U(\mathfrak g),P)$-module homomorphism (so in
particular, the Verma modules share the same infinitesimal character), then in view of
\eqref{twisting}, and assuming further that both $V(\mathbb E')$ and $V(\mathbb
F')$ enjoy the same and unique infinitesimal character in $V(\mathbb E\otimes
\mathbb W)$ and $V(\mathbb F\otimes \mathbb W)$, respectively, we
obtain the composed homomorphism
\begin{equation}\label{translation}
V(\mathbb F')\to V(\mathbb F\otimes \mathbb W)=V(\mathbb F)\otimes W^*\to
V(\mathbb E)\otimes \mathbb W^*= V(\mathbb E\otimes \mathbb W) \to V(\mathbb
E')
.\end{equation}
We talk about \emph{twisting the homomorphism} $\Ph$ by tensoring it with the identity
on $\mathfrak g$-module $\mathbb W^*$.

A difficult question remains, how to recognize whether the \emph{translated 
morphism} $V(\mathbb F')\to V(\mathbb E')$ is nontrivial.

\subsection{The Jantzen-Zuckermann translation principle} Before we explain
why the shapes of the de Rham complexes in \eqref{1|2} - \eqref{3|3} happen
to be the general patterns for all infinitesimal characters, let us focus on
the action of the Weyl group $W_{\mathfrak g}$ on the weights. 

For each $\mathfrak p$-dominant weight $\al$, there is exactly one $s\in
W_{\mathfrak g}$ such that $\al+\rho = s\cdot(\la+\rho)$ for a uniquely defined
$\mathfrak g$-dominant weight $\la+\rho$, i.e., $\la+\rho$ 
sits in the closed dominant
Weyl chamber. If $\la$ itself is $\mathfrak g$-dominant, then we say that
the infinitesimal character $\xi_\al$ is \emph{regular}. If $\la+\rho$ sits in a
wall (or intersection of several walls) of the dominant chamber, we call
$\xi_\al$ \emph{singular} (or more precisely $k$-singular, if sitting on 
intersection of $k$ walls).

Starting with a $\mathfrak g$-dominant $\la$, we obtain the so called Hasse
diagram of the orbit of the subgroup $W_{\mathfrak p}$ of those $s\in
W_{\mathfrak g}$ with the affine action producing $\mathfrak p$-dominant
weights. This Hasse diagram is independent of the chosen dominant weight $\la$,
see \cite[Section 3.2.18]{parabook} for recipes how to get it. 
If $\la$ is not
dominant, but $\la+\rho$ is, then still the affine action of $W_{\mathfrak p}$
produces some $\mathfrak p$-dominant weights on its orbit, which all 
appear with
$2^k$ repetitions, if the infinitesimal character is $k$-singular.

The length of $s$ is defined as the least number of simple
reflections composed to built $s$.  Looking at our de Rham patterns, the
lengths of such $s$ is the number of transpositions of neighbors in the
permutation of the numbers and this also labels the columns there (e.g., going
from zero to nine in \eqref{3|3}).

For a moment, let us come back to the regular infinitesimal characters and let us
write $\mathbb E_{\al}$ or $\mathbb W_\mu$ for the modules with $\mathfrak p$
or $\mathfrak g$-dominant 
highest weights $\al$ or $\mu$, respectively. 
Following \cite{Zu,BJ}, we define two functors on 
$\mathfrak U(\mathfrak g)$-modules, which split into direct sums of
components with respect to the infinitesimal characters. These include 
our (generalized) Verma modules. We shall write $p_\la$ for the projection
of such modules to the component with the infinitesimal character $\xi_\la$. 
Consider two $\mathfrak g$-dominant weights
$\la$, $\mu$, and define the \emph{translation functors}
\begin{align}\label{JaZu-adjoints} 
\ph^\la_{\la+\mu} &= p_{\la+\mu}\o(-\otimes \mathbb W_\mu)\o p_\la 
\\
\psi^{\la+\mu}_\la &= p_\la\o (-\otimes (\mathbb W_\mu)^*)\o p_{\la+\mu}
\end{align}
where the action on morphisms is given by twisting by the identity in the
tensor product.

Actually, the same construction works if $\mu$ is $\mathfrak g$-dominant,
while $\la$ is a $\mathfrak p$-dominant weight with a singular infinitesimal
character. We say that the weights 
$\la$ and $\la'$ are \emph{equi-singular} if their
singular character is represented by a weight on the same (intersection of)
wall(s) of the dominant Weyl chamber. In particular, all weights with
regular infinitesimal character are considered equi-singular in this sense. 

\begin{thm} Consider a $\mathfrak g$-dominant weight $\mu$ and a $\mathfrak
p$-dominant weight $\la$ such that $\la+\rho$ is in the closed dominant Weyl
chamber.

\noindent (1) The functor $\psi^{\la+\mu}_\la$ is left adjoint to $\ph^\la_{\la+\mu}$.

\noindent (2) If the weights $\la$ and $\la+\mu$ are equi-singular, then 
$$
\psi^{\la+\mu}_\la(V(\mathbb E_{s\cdot(\la+\mu)}))= V(\mathbb
E_{s\cdot\la}), \quad \ph^{\la}_{\la+\mu}(V(\mathbb
E_{s\cdot\la}))=V(\mathbb E_{s\cdot(\la+\mu)})
,$$
whenever $s\cdot \la$ is $\mathfrak p$-dominant.
\end{thm}

\begin{proof} Since $\mathbb W_\mu$ is finite dimensional, there is the
tautological isomorphism for all $\mathfrak p$-modules $\mathbb E$ and
$\mathbb F$,
\begin{equation}\label{taut-iso}
\operatorname{Hom}_{(\mathfrak U(\mathfrak g),P)}(V(\mathbb F)\otimes\mathbb
(W_\mu)^*,V(\mathbb E)) = \operatorname{Hom}_{(\mathfrak U(\mathfrak g),P)}
(V(\mathbb F),V(\mathbb E)\otimes\mathbb W_\mu)
.\end{equation}

As we know, only the summand $p_\la(V(\mathbb
E_{s\cdot(\la+\mu)})\otimes(\mathbb W_\mu)^*)$ can contribute to
$\operatorname{Hom}_{(\mathfrak U(\mathfrak g),P)}(V(\mathbb
E_{s\cdot(\la+\mu)})\otimes(\mathbb W_\mu)^*,V(\mathbb E_{s'\cdot\la})$ and
similarly only $p_{\la+\mu}(V(\mathbb
E_{s'\cdot\la})\otimes\mathbb W_\mu)$ can contribute to
$\operatorname{Hom}_{(\mathfrak U(\mathfrak g),P)}(V(\mathbb
E_{s\cdot(\la+\mu)}),V(\mathbb E_{s'\cdot\la})\otimes \mathbb W_\mu)$. Thus,
we have arrived at the requested natural equivalence
\begin{align*}
\operatorname{Hom}_{(\mathfrak U(\mathfrak
g),P)}\bigl(\psi^{\la+\mu}_\la&\bigl(V(\mathbb
E_{s\cdot(\la+\mu)})\bigr),V(\mathbb E_{s'\cdot\la})\bigr)
\\
&\simeq \operatorname{Hom}_{(\mathfrak U(\mathfrak
g),P)}\bigl(V(\mathbb
E_{s\cdot(\la+\mu)}),\ph^\la_{\la+\mu}\bigl(V(\mathbb E_{s'\cdot\la})\bigr)\bigr)
.\end{align*}

The second claim is more difficult to prove. We present a quick sketch
only. As shown in \cite{Zu}, for equi-singular $\la$ and $\la+\mu$, the
functors $\psi^{\mu+\la}_\la$ and $\ph^\la_{\la+\mu}$ are mutually inverse
natural equivalences. Clearly, the $\mathfrak p$-dominant weights $\la$ and
$\la+\mu$ appear at the same position in the Hasse diagram, the
infinitesimal character is shared by the entire Hasse diagram, and thus
$p_\la$ is identity on every $V(\mathbb E_{s\cdot\la})$. Next, as discussed
above, the tensor product $V(\mathbb E_{s\cdot\la}\otimes\mathbb W_\mu)$ 
(as a $\mathfrak g_0$-module) decomposes,
$$
V(\mathbb E_{s\cdot\la}\otimes\mathbb W_\mu)=
\mathfrak U(\mathfrak g_-)\otimes((\mathbb E_{s\cdot \la})^*\otimes \mathbb
W_\mu) = \mathfrak U(\mathfrak g_-)\otimes (\oplus_{j=1}^k\mathbb W_{\nu_j})
= \oplus_{j=1}^k V(\mathbb W_{\nu_j})
.$$ 
The weights $\nu_j$ in the sum appear with multiplicities which can be
computed explicitly, e.g., by means of the Klimyk formula. Finally, the
projection $p_{\la+\mu}$ selects only those with the infinitesimal character
$\xi_{\mu+\la}$. 

Summarizing, the value of $\ph^\la_{\la+\mu}$ on a generalized Verma module
is a sum of generalized Verma modules. Swapping $\mathbb W_\mu$ and $\la$
with $(\mathbb W_\mu)^*$ and $\la+\mu$, we get the same claim for
$\psi^{\la+\mu}_\la$. Now, we know that $\psi^{\la+\mu}_\la\o\ph^\la_{\la+\mu}$
is naturally equivalent to identity and, thus, the values can always consist
of one Verma module only. Certainly, $\nu=s\cdot(\la+\mu)$ appears among the
weights $\nu_j$, and it must appear with multiplicity one. As a result,
$$
\ph^\la_{\la+\mu}(V(\mathbb E_{s\cdot\la})) = V(\mathbb E_{s\cdot(\la+\mu)})
.$$
Similarly, we understand the functor $\psi^{\la+\mu}_\la$, replacing $\mu$ by $-\mu$ and
$\la$ by $\mu+\la$.
\end{proof}

\subsection{Back to examples} As a direct consequence of the theorem we
understand that the de Rham pattern is copied for each $\mathfrak g$-dominant
weight. With all non-trivial morphisms at exactly the same positions. The
order of the morphisms is easily computed as the difference of the actions
of the grading element, i.e., using the formulae \eqref{ell_fundamental} or
\eqref{ell_ntuples} (possibly working with the weights $\la+\rho$, since
constant changes do not matter).

To see a $1$-singular example, take the $\mathfrak p$-dominant weight
$\la=(21|10)$ as the left most weight in the pattern \eqref{2|2} (remember,
the notation is describing rather $\la+\rho$, which sits in the wall of the
dominant chamber), and apply the same permutations as in \eqref{2|2} again. 
We arrive at (weights which are not $\mathfrak p$-dominant are replaced by
crosses) 
\begin{equation}\label{2|2-singular} 
\xymatrix@R=2mm@C=4mm{
\wgt{21}{10}
&&{\ \ \x\ \ }
&&\wgt{10}{21} \ar@{=}[dl]
\\
&\wgt{21}{10} \ar@{=}[ul]
&&\wgt{10}{21} \ar@{..>}@(ul,dr)[ll]
\\
&&{\ \ \x\ \ }
}
\end{equation}
where the dotted homomorphism corresponds to the second order composition in
the central diamond in the de Rham pattern, and it corresponds to the
conformally invariant Laplacian, i.e., the Yamabe operator, on densities
with the right weights. 

By the results of Enright and Shelton, \cite{EnSh}, there are bijective
correspondences between the patterns for singular infinitesimal characters
and patterns for regular characters in lower dimensional geometries, we
shall not go into details here. For example, the pattern in
\eqref{2|2-singular} coincides with the regular one for one-dimensional
projective geometry. There, the de Rham consists of one morphism only,
exactly as seen in \eqref{2|2-singular}.

Similarly, the only 2-singular pattern for the $(2,2)$ Grassmannian will
consists of four equal $\mathfrak p$-dominant weights and there are no
no-trivial morphisms there. All the 2-singular patterns for the 
$(3,3)$ Grassmannian will be again of the same shape as the one-dimensional
projective de Rham (there will be two groups of four equal $\mathfrak p$-dominant
weights, with one non-trivial homomorphism between them).

Actually, our aim is to extend this algebraic translations to the realm of
curved Cartan geometries in the next section. For this endeavor, the crucial
observation in \cite{ES} was, that actually the above considerations allow
for translations based on many other weights of $\mathbb W_\mu$ than the
highest and lowest ones. We formulate this observation as two propositions:

\begin{prop}[Proposition 9 in \cite{ES}]\label{transprop1}
Suppose that $V({\Bbb E})$ and $V({\Bbb F})$ have the same
infinitesimal character. Suppose that $V({\Bbb E}^\prime)$ and $V({\Bbb
F}^\prime)$
have the same infinitesimal character. Let ${\mathbb W}$ be a 
finite-dimensional
irreducible representation of $G$ and suppose that
\begin{itemize}
\item $V({\Bbb F}^\prime)$ occurs in the composition series for
      $V({\Bbb F}\otimes{\mathbb W})$ and has distinct infinitesimal 
character from all other factors;
\item $V({\Bbb E}^\prime)$ occurs in the composition series for
      $V({\Bbb E}\otimes{\mathbb W})$ and has distinct infinitesimal 
character from all other factors.
\end{itemize}
It follows that $V({\Bbb F})$ occurs in the composition series for
$V({\Bbb F}^\prime\otimes{\mathbb W}^*)$ and that $V({\Bbb E})$ occurs in the
composition series for $V({\Bbb E}^\prime\otimes{\mathbb W}^*)$. We suppose
further that
\begin{itemize}
\item all other composition factors of $V({\Bbb F}^\prime\otimes{\mathbb
W}^*)$ have infinitesimal character distinct from $V({\Bbb F})$;
\item all other composition factors of $V({\Bbb E}^\prime\otimes{\mathbb
W}^*)$ have infinitesimal character distinct from $V({\Bbb E})$.
\end{itemize}
Then translation gives an isomorphism
$$\operatorname{Hom}_{({\frak U}({\frak g}),P)}(V({\Bbb F}),V({\Bbb
E}))\simeq
  \operatorname{Hom}_{({\frak U}({\frak g}),P)}(V({\Bbb F}^\prime),V({\Bbb E}^\prime))$$
whose inverse is given by translation using~${\mathbb W}^*$.
\end{prop}
\begin{proof} Straightforward, using the tautological isomorphisms
\eqref{taut-iso} and the above argumentation. 
\end{proof}

Actually, sometimes there are also one-way translations producing
non-trivial homomorphisms in the less singular patterns from the more
singular ones. For example, in the case of the $(2,2)$ Grassmannian, we
might start with the identity morphism in the only 2-singular pattern,
produce the morphism corresponding to the first order Dirac operator on the
basic spinors (which is in the other 1-singular pattern there), as well as
the second order morphism from \eqref{2|2-singular}. But, the 4th-order
morphism corresponding to the Paneitz operator in the de Rham pattern cannot
be translated from anything else. 

Such one-way translations are based on the
following extension of the previous proposition. Recall, there is the scalar
action of the grading element on irreducible $\mathfrak p$-modules $\mathbb
E$. We write $\al(\mathbb E)$ for this constant, now. 

Notice, that the
difference of these scalars determines the order of the prospective
homomorphisms.

\begin{prop}\label{transprop2}
Let $\Phi:V(\mathbb E)\to V(\mathbb F)$ be a nontrivial homomorphism of 
Verma modules,
and let $\mathbb W$ be an irreducible finite dimensional  $G$-module.

Suppose  that there are irreducible $\mathfrak p$-modules
$\mathbb E_1,\mathbb E_2,\mathbb F_1,\mathbb F_2$ such that:
	
\noindent (i) 
	$\mathbb E\otimes\mathbb W=\mathbb E_1\oplus
\mathbb E_2\oplus\mathbb E';\;\mathbb F\otimes\mathbb W=\mathbb F_1\oplus \mathbb F_2\oplus \mathbb F';$
	
\noindent (ii) Verma modules $V(\mathbb E_1),V(\mathbb E_2),V(\mathbb F_1), V(\mathbb F_2)$ have the
same  infinitesimal character;
	
\noindent (iii)
	all pieces in the composition series for $V(\mathbb E'), V(\mathbb F')$ have
different infinitesimal characters.
	from those in (ii);
	
\noindent (iv) $ \al(\mathbb E_1)< \al(\mathbb E_2)$,
$\al(\mathbb F_1)>\al(\mathbb F_2)$  
 and  
	$V(\mathbb F)$ splits off from   $V(\mathbb F_1\otimes\mathbb W^*).$ 
	
	\vskip 1mm\noindent
	If there is no nontrivial homomorphism from $V(\mathbb E_2)$  to
$V(\mathbb F),$ 
	then the translated homomorphism
	$$\hat{\Phi}: V(\mathbb E_1)\to V(\mathbb E\otimes\mathbb W)\to  V(\mathbb F\otimes\mathbb W)\to
V(\mathbb F_1)$$ is nontrivial.    
\end{prop}

\begin{proof}
Our assumptions imply that $V(\mathbb E_1)$ embeds to
$V(\mathbb E\otimes\mathbb W)$ 
and $V(\mathbb F\otimes\mathbb W)$ projects to $V(\mathbb F_1)$.
Hence the translated homomorphism $\hat{\Phi}$  is well defined.
$V(\mathbb F)$ splits off from  $V(\mathbb F_1\otimes\mathbb W^*)$, 
hence the composition
$$
V(\mathbb E)\stackrel{\Phi}{\rightarrow}  V(\mathbb F)\to V(\mathbb F_1\otimes\mathbb W^*).
$$ 
is nontrivial. Based on \eqref{taut-iso}, this is
equivalent to the fact that 
$$
V(\mathbb E\otimes\mathbb W)\stackrel{\Phi\otimes Id_\mathbb W}{\longrightarrow}
V(\mathbb F\otimes\mathbb W)\to V(\mathbb F_1)
$$
is nontrivial.

From (iii) it follows that
$V(\mathbb E\otimes\mathbb W)=V(\mathbb E_1\oplus\mathbb E_2)\oplus V(\mathbb E').$ 
Using (ii) and (iii), it is clear that also the composition
$$
V(\mathbb E_1\oplus \mathbb E_2)\stackrel{\Phi\otimes Id_\mathbb W}{\longrightarrow}
V(\mathbb F\otimes\mathbb W)\to V(\mathbb F_1)
$$
is nontrivial.

In case that the composition
$V(\mathbb E_1)\to V(\mathbb E_1\oplus\mathbb E_2)  \to V(\mathbb F)\to  V(\mathbb F\otimes\mathbb W)\to
V(\mathbb F_1)$ 
would be trivial, it follows that there is a nontrivial homomorphism 
$V(\mathbb E_2)$  to $V(\mathbb F_1),$ which is a contradiction. 
\end{proof}

\section{The curved translation principle}

\subsection{Cartan connections} 
Finally, we come to the curved Cartan geometries, modeled over the Klein's
homogeneous spaces $G/H$. This concept generalizes the affine connections on manifolds $M$,
realized as the sum of soldering forms and principle connections on the
linear frame bundle on $M$, i.e., the bundles of frames of $TM$. The affine
$\mathbb R^n$, as the homogeneous space $\operatorname{Aff}(n,\mathbb
R)/\operatorname{GL}(n,\mathbb R)$ is the relevant Klein model for the
affine connections.

The reader may find all the relevant background on Cartan connections in
\cite[Section 1.5]{parabook} or \cite{Sha}. 

\begin{definition*} 
Let $G$ be a (finite dimensional) Lie group, $H$ its closed subgroup. 
A Cartan connection of type $G/H$ is a principal fiber bundle $\mathcal G\to
M$ with structure group $H$, equipped by a 1-form $\om\in\Om^1(\mathcal
G,\mathfrak g)$, satisfying all the properties of the Maurer-Cartan form on
$G$, which make sense:

\noindent (i) The one-form $\om$ is $H$-equivariant, i.e., $(r^h)^*\om =
\operatorname{Ad}(h^{-1})\o \om$ for all $h\in H$.

\noindent (ii) The one-form $\om$ reproduces fundamental vector fields on
$\mathcal G$, i.e., $\om(\ze_X(u))= \om^{-1}(X)(u)$, for all $u\in\mathcal
G$, $X\in\mathfrak h$, $\ze_X(u)=\frac\partial{\partial
t}_{|t=0}u\cdot\operatorname{exp}tX$.

\noindent (iii) The one-form $\om$ is an absolute parallelism, i.e.,
$\om(u):T_u\mathcal G\to \mathfrak g$ is a linear isomorphism.
\end{definition*}

The morphisms between the Cartan connections $\om$ and $\om'$ on
principal bundles $\mathcal G$
and $\mathcal G'$, with the same structural group $H$, 
are principle fiber bundle
morphisms $\phi:\mathcal G\to \mathcal G'$ (over the identity on the group
$H$), satisfying $\ph^*\om'=\om$. 

The group of automorphisms of a given Cartan connection of type $G/H$ is
always a finite dimensional Lie group whose Lie algebra is a subalgebra in
$\mathfrak g$. In particular, its dimension is bounded by the dimension of
$G$, see \cite[Section 1.5.11]{parabook}. 

In practice, the geometries are mostly defined by some simpler infinitesimal
data, for example a G-structure, i.e., reduction of the structure group
$GL(n,\mathbb R)$ of
the linear frame bundle to a closed subgroup $G$. Riemannian manifolds and
conformal Riemannian manifolds are typical examples. The (normalized)
natural Cartan connection is then a result of a construction (a so called
prolongation).  The theory and many examples of important
Cartan geometries are discussed in great detail in \cite[Chapters 1, 4, and 5]{parabook}. 

The existence of the automorphisms is closely related to the
\emph{curvature} of the Cartan connection $\om$, 
the two-form $K\in\Om^2(\mathcal
G; \mathfrak g)$, $K=d\om + \frac12[\om,\om]$. 

Of course, the Maurer-Cartan
equations say that $K=0$ if $\om$ is the Maurer-Cartan form on the Lie group
$G$. 

There is the well known theorem that the Cartan geometry is locally
isomorphic to its Klein's model, if and only if $K$ vanishes, see
\cite[Section 1.5.2]{parabook}. 

\subsection{Natural bundles} The construction of the
homogeneous bundles $\mathcal E$ from the $H$-modules $\mathbb E=\mathcal
E_O$ from 
\ref{homogeneous-bundles}, extends directly to a functor on the category of
principal bundles and their morphisms, by the very same construction of the
associated bundles. We shall now write $\mathcal E$ for the functor mapping
the principal bundles $\mathcal G\to M$  to the associated bundle 
$\mathcal G\x_H \mathbb E\to M$, with the obvious action on morphisms.

Any $H$-module homomorphism provides a natural transformation of the
corresponding functors. 

Notice that exactly as in \ref{1.2}, the sections of the  natural bundles
are identified with the $H$-equivariant functions in $C^{\infty}(\mathcal
G,\mathbb E)$. 

The most classical examples of Cartan connections 
are the affine connection on a manifolds. Due to 
the reductive structure of the Klein's model (i.e., $G$ is the affine group in
dimension $n$, $H=\operatorname{GL}(n,\mathbb R)$ 
and the horizontal directions
$\mathbb R^n$ form an $H$-submodule), the natural vector bundles are 
essentially only
components of tensor bundles over the underlying manifolds. 
The Cartan connection then splits as
$\theta+\gamma$, the soldering form and the linear connection form, while
its curvature $K=T+R$ splits naturally into the torsion and
curvature of the affine connection. Finally, there is the well known 
Schouten's reduction theorems 
saying that all invariant differential operators are in this setting
obtained via covariant derivatives of the arguments, the curvature and
torsion, and invariant algebraic operations (see \cite[Chapter 28]{KMS}). 

We should also recall that
viewing the sections of natural bundles as functions in $C^\infty(\mathcal
G,\mathbb E)_{\operatorname{GL}(n,\mathbb R)}$, the covariant derivative
with respect to an affine connection is
simply given by differentiating in the direction of the constant fields
$\om^{-1}(X)$ for $X\in \mathbb R^n$ (notice, at a fixed frame $u$ in the
linear frame bundle $\mathcal G$ of the base manifold $M$, $X$ is then
identified with a tangent vector on $M$). 

For general Cartan geometries, the existence of natural covariant
derivatives on natural bundles is a subtle, but completely algebraic,
question with answers completely inherited from the Klein's models, see
\cite[Section 1.5.6]{parabook}.  The Riemannian geometry (with the three
different possible homogeneous models - the Euclidean, hyperbolic and
spherical spaceforms) is a special example of a geometry with 
reductive model ensuring the
unique normalized connection, the Levi Civita connection. 

Just in the case of $G$-modules, i.e., dealing with \emph{tractor bundles},
there is always the natural linear connection induced by the Cartan
connection itself on all of them. This follows from exactly the same 
arguments as in
\ref{homogeneous-bundles}, see also \cite[Section 1.5.7]{parabook}. 

The \emph{adjoint tractor bundle} $\mathcal A = \mathcal G\x_H\mathfrak g$
provides an extremely important example. The short exact sequence of
$H$-modules $0\to \mathfrak h\to \mathfrak g\to \mathfrak g/\mathfrak h\to 0$
gives rise to short exact sequence of natural bundles
$$
0\to\mathcal G\x_H \mathfrak h \to \mathcal A\to TM\to 0 
.$$ 
A straightforward check reveals that the curvature of the Cartan connection
descends to a two-form in $\Om^2(M;\mathcal A)$. See \cite[Section
1.5.7]{parabook} for further properties and details.  

\subsection{Grassmannian geometries} The Cartan geometries modeled over the
$(m,n)$-Grassmannians are examples of G-structures defined by the reduction
of the structure group $S(GL(m,\mathbb R)\x GL(n,\mathbb R))\subset
GL(m+n,\mathbb R)$, as explained in \ref{Grass_ex}. We call them almost
Grassmannian geometries and they are equivalently defined by identifying $TM$
with tensor product of the two auxiliary vector bundles $EM$ and $F^*M$ of
dimensions $m$ and $n$, respectively, together with the identification of
the top degree forms on $E$ and $F^*$, again as discussed in \ref{Grass_ex}. 

As before, we may assume $m\le n$. The projective geometries correspond to
$m=1$ and the Cartan curvature has got values in $\mathfrak p$, i.e., there
is no torsion. If $m=n=2$, we deal with the split signature conformal
Riemannian geometries in dimension $4$, and there are two curvature
components there (again no torsion). 

If $m=2<n$, then there is one torsion
and one curvature there and the geometries without the torsion are higher
dimensional analogues of the self-adjoint 4-dimensional conformal structures.
Dealing with the quaternionic real form of the same complexified algebras we
arrive at the (almost) quaternionic geometries. 

If $2<m\le n$, the almost Grassmannian geometries come with two torsion
components. The special cases of $m=n$ are of special interest, since they
are another promising generalization of the 4-dimensional conformal
geometries. 

In all the above cases, the identification of the top degree forms on $EM$
and $FM$ implies that the classical results from the tensorial invariant
calculus may be employed for the two sets of abstract in indices describing
fields in $\mathcal E^{A\dots BC'\dots D'}_{E\dots FG'\dots H'}[w]$.  

\subsection{Invariant operators and jet prolongations}
In the Klein's world of geometric analysis, the invariant operators can be viewed as 
natural transformations between the relevant jet prolongations of the
homogeneous bundles. This does not make sense now, because the
existence of curvature excludes or reduces the existence of 
(auto)morphisms of the Cartan geometries. 

It seems there are two options to move forward: either to exploit the
concept of the so called gauge-natural operators, see \cite[Chapter 12]{KMS}, i.e., we
would work over the category of principal bundles and gauge-natural bundles, 
and add the Cartan connection
$\om$ to the arguments of the operators, or we rather seek
ways, how to extend the operators at the Klein's model to the general cases.

In the first case, we mostly 
drastically reduce the choice of the natural bundles. Thus, we shall focus
on the second approach only.  
 
Perhaps the first idea should be to exploit the equivalence between the
invariant linear operators and $H$-module morphisms shown in Proposition
\ref{prop.1.3}. This happens to be a bit tricky, though.

Fortunately, we may mimic the idea of expressing the directional derivatives
via the actions of the constant fields $\om^{-1}(X)$ on the $H$-equivariant
functions. Indeed, for each Cartan connection $\om$ we obtain the so called
\emph{fundamental derivative} $D^\om$, which is an operation
\begin{equation}\label{fundamental-derivative}
C^\infty(\mathcal G, \mathbb E)_H\ni \si \mapsto D^{\om}(\si) =  
(X\mapsto \om^{-1}(X)\cdot \si)\in C^\infty(\mathcal G,\mathfrak g^*\otimes\mathbb
E)_H
,\end{equation} 
for all $X\in\mathfrak g$. Thus, the fundamental derivative is a natural differential 
operation mapping sections of $\mathcal E$
to sections of $\mathcal A^*\otimes \mathcal E$, which may be iterated. 
See \cite[Section
1.5.8]{parabook} for details and further properties.

Of course, there is a lot of redundancy there, since the derivatives in the
directions of $\om^{-1}(X)$, with $X\in\mathfrak h$, act algebraically. Let us
look at the situation at the level of $H$-modules. The value of the
fundamental derivative $D^\om \si$, together with the value $\si$, can be
understood as a couple $(v,\ph)\in\mathbb E\oplus \mathfrak
g^*\otimes\mathbb E$. The action of $h\in\mathfrak h$ is
$$
h\cdot(v,\ph) = \bigl(h\cdot v, X\mapsto
h\cdot\ph(\operatorname{Ad}_{h^{-1}}X)\bigr)
.$$
Comparing this with the natural action on the first jet prolongation
$J^1\mathbb E$, we can see that actually $J^1\mathbb E$ is naturally
embedded in $\mathbb E\oplus\mathfrak g^*\otimes\mathbb E$ as the
$H$-submodule consisting of couples $(v,\ph)$ with 
$\ph(Z)=-Z\cdot v$, for all $Z\in\mathfrak h$ (which perfectly mimics the
fact that $D^{\om}\si(Z) = -Z\cdot \si$ for such $Z$). 

Thus, the fundamental
derivative provides the universal first jet prolongation of sections.
If we choose a complementary subspace $\mathfrak g_-$ identified with the
quotient $\mathfrak g/\mathfrak h$, we can restrict $D^\om$ to $\mathfrak
g_-$. Then, exactly as for the Klein's model, we arrive at
$$
J^1\mathcal E\simeq \mathcal G\x J^1\mathbb E
,$$
and the universal differential operator $\si\mapsto j^1\si$ defined by the
restriction of the fundamental derivative to $\mathfrak g_-$. 

In particular, we have verified that each $H$-module homomorphism defining
an invariant first order linear operator on the Klein's model $G/H$ 
directly extends to an invariant
linear differential operator on the category of Cartan connections of the
type $G/H$.

\subsection{Higher order jets}\label{3.4}
We might repeat the argumentation from the previous paragraph and find the
second order prolongation module 
$$
J^2\mathbb E = \mathbb E + (
(\mathfrak g/\mathfrak h)^*\otimes \mathbb E) + (S^2(\mathfrak g/\mathfrak
h)^*\otimes \mathbb E)
$$ 
naturally embedded in the module
$$\mathbb E \oplus (\mathfrak g^*\otimes\mathbb E)\oplus (\mathfrak
g^*\otimes\mathfrak g^*\otimes \mathbb E)
$$ 
via the iterated action of the
fundamental derivative (notice, the first module is given as composition
series, while the second one is a direct sum). Thus, we arrived again at the
identification $J^2\mathcal E=\mathcal G\x_H J^2\mathbb E$ for all
$H$-modules $\mathbb E$. Now, 
every second order invariant (linear) operator on the 
Klein's
model $G/H$ corresponds to a $\mathfrak h$-module homomorphism $\Phi$ on the
relevant jet prolongations $J^\mathbb E$. Thus, each such operator 
extends to an invariant operator between the corresponding
natural bundles in the entire category of Cartan connections of this type,
by exploiting the iterated fundamental derivative and using the same $\Phi$. 

However, this fails for orders bigger than two. The reason is explained in
\cite[Section 1.5.10]{parabook} -- the iterated fundamental derivative
$D^\om$ always defines and injective universal operator 
$$
J^r\mathcal E\to \oplus_{j=0}^r S^j\mathcal A^*\otimes\mathcal E
,$$
but for $r\ge 3$, $J^r\mathcal E$ is not naturally identified with the
associated bundle $\mathcal G\x_H J^r\mathbb E$.

We may iterate $J^1(\dots J^1(J^1\mathcal E))$, and although this involves unnecessary
redundancies, these can be removed by the following well known categorical
construction.

\begin{definition*}
The semiholonomic jets $\bar J^k\mathcal E$ are inductively 
defined as the equalizer of all the natural 
projections $J^1(\bar J^{r-1}\mathcal E)\to \bar J^{r-1}\mathcal E$,
starting with $\bar J^1\mathcal E=J^1\mathcal E$. 

For every $H$-module $\mathbb E$, we define its semiholonomic jet
prolongation as the $H$-module $\bar J^r(G\x_H\mathbb E)_O$, i.e., the fiber
over origin of the semiholonomic prolongation of $\mathcal E$ over $M=G/H$.

By the very construction, $\bar J^r\mathcal E\simeq \mathcal G\x_{H}\bar J^r$, and the iterated fundamental derivatives define the
universal differential operator $\mathcal E \to \bar J^r\mathcal E$.
\end{definition*}   

In particular, $\bar J^2\mathcal E$ is the equalizer of the two natural
projections appearing as the value of the functor $J^1$ on the projection
$J^1E\to \mathcal E$, and the projection $J^1(J^1\mathcal E) \to J^1\mathcal
E$.  

The modules defining the semiholonomic jet prolongations as natural bundles are 
$$
\bar J^r\mathbb E= \mathbb E + (\mathfrak g/\mathfrak h)^*\otimes
\mathbb E + \dots + \otimes^r(\mathfrak g/\mathfrak h)^*\otimes
\mathbb E 
$$
which is a composition series (the right ends are $H$-submodules) with quite
wild action of $H$.

Obviously we have got a one-way analogy of the Proposition \ref{prop.1.3}:

\begin{prop*}
Each non-zero $H$-module homomorphism $\Ph :\bar J^k\mathbb
E\to \mathbb F$ defines invariant linear differential operators 
$D:\Ga(\mathcal E)\to
\Ga(\mathcal F)$ of order at most $k$.
\end{prop*}

Notice that the opposite implication fails in general because the image of
the universal operator $\mathcal E\to\bar J^k\mathcal E$ is an algebraic
subvariety in the target and there are counterexamples of operators defined
by a morphism on the image of the universal operator (restricted to the
fiber over the origin in the model), but not extending to a genuine $H$-module
morphism on the entire $\bar J^k\mathbb E$. We shall comment more on this
phenomenon later.
 
\subsection{Semiholonomic induced modules}\label{3.5}
Exactly as in the Klein's model case, we better look at the dual picture.
Here we follow \cite{ES}, where the basic concepts were defined first. Although only the conformal Riemannian structures and the relevant operators
and (semiholonomic) Verma modules were discussed in
\cite{ES}, many steps can be employed in general, without any modification. 

As we have seen, the role of the
left invariant vector fields are for Cartan geometries played by the constant
vector fields $\om^{-1}(X)\in\mathcal X(\mathcal G)$, $X\in \mathfrak g$.  
Differentiating
the $H$-equivariant functions $\si:\mathcal G\to \mathbb E$ in the 
direction of $\om^{-1}(X)$ yields the fundamental derivative of the
sections. More precisely, if $X\in\mathfrak h$, then $(\om^{-1}(X)\cdot\si)(u) =
-X\cdot (\si(u))$ by the equivariance and, thus, the genuine differential
parts are again in the quotient $\mathfrak g/\mathfrak h$, 
thus corresponding to
derivatives of the sections in directions in $T_OM$.

Next, consider again a `word' $X_1X_2\dots X_k$ of elements in 
$\mathfrak g$ and
the corresponding differential operator $\si\mapsto
\om^{-1}(X_1)\o\om^{-1}(X_2)\o\dots\o\om^{-1}(X_k)\cdot\si(u)$ on the
functions, evaluated in a frame $u\in\mathcal G$. 

We may consider this operation as defined on the tensor algebra $T(\mathfrak
g)$ and again, there is the ideal $\mathcal I$ in $T(\mathfrak g)$ generated by the
expressions $X\otimes Y-Y\otimes X - [X,Y]$, with $X, Y\in \mathfrak g$, but
at least one of them in $\mathfrak h$, which acts trivially. This is the
consequence of the fact, that the curvature of the Cartan connection $\om$
vanishes if one of the arguments is vertical.

The resulting quotient (left and right) $\mathfrak g$-module $\bar{\mathfrak
U}(\mathfrak g)
= T(\mathfrak g)/\mathcal I$ is called
the \emph{semiholonomic universal enveloping algebra} 
of the Lie algebra $\mathfrak g$. 

Next we want to understand the linear forms in the dual of the semiholonomic 
jet modules
$(\bar J^k\mathbb E)^*$. Exactly as with the induced modules, 
we differentiate functions also in the vertical
directions, and our values are in $\mathbb E$.  Thus we consider the
tensor product 
$$
\bar V(\mathbb E)=\bar{\mathfrak U}(\mathfrak g)\otimes_{\mathfrak
U(\mathfrak h)} \mathbb E^*
.$$ 
The space $\bar V(\mathbb E)$ clearly enjoys the
structure of a $(\mathfrak g,H)$-module (and $(\bar{\mathfrak U}(\mathfrak
g),H)$-module), and it is called the \emph{semiholonomic induced module} 
for the
$H$-module $\mathbb E$.

\begin{prop}\label{prop.3.6}
The induced module $\bar V(\mathbb E)$ is the space of all linear forms on
$\bar J^\infty\mathbb E$ which factor through some $\bar J^k\mathbb E$, 
i.e., depend on finite number of derivatives. There is the natural
surjection $\bar V(\mathbb E)\to V(\mathbb E)$.
\end{prop}
\begin{proof}
As in the induced modules case, the claim follows from the construction of $\bar V(\mathbb E)$ and 
the fact that
choosing a complementary vector subspace to $\mathfrak h$ in $\mathfrak g$,
we can decompose all letters in our words $X_1\dots X_k$ above and, by the
equalities enforced by living in the quotient by the ideal, we may ``bubble'' 
the letters in $\mathfrak h$ to the very right. Once there, they act 
algebraically and, thus,
tensorizing over $\mathfrak U(\mathfrak h)$ we remove just all redundancies.
\end{proof}

Obviously again, $\mathbb E^*$ injects into $\bar V(\mathbb E)$, generates this
$\mathfrak g$-module, and there is the natural filtration
$$
\mathbb E^*=\bar V_0(\mathbb E)\subset \bar V_1(\mathbb E)\subset \dots
\subset \bar V_k(\mathbb E)\subset \dots\subset \bar V(\mathbb E)
$$
inherited from the filtration on $T(\mathfrak g)$.

Next, assume that there is a fixed complementary subalgebra $\mathfrak g_-\simeq
\mathfrak g/\mathfrak h$ to
$\mathfrak h\subset \mathfrak g$. Then the Poincar\'e-Birkhoff-Witt procedure
reveals that the graded semiholonomic universal algebra equals 
$\operatorname{gr}\bar {\mathcal U}(\mathfrak g)=T(\mathfrak g_-)$ as vector
space, while the graded semiholonomic induced module $\bar V$ is then, as a
vector space, isomorphic
to $\bar {\mathcal U}(\mathfrak g_-)\otimes_{\mathbb R} \mathbb E^*$. 
Moreover, there is the the following commutative 
diagram of short exact sequences:
\begin{equation}\label{jetdiagram}
\xymatrix@R=5mm@C=10mm{
0 \ar[r] 
& \bar V_{k-1}(\mathbb E) \ar[r] \ar[d]
& \bar V_k(\mathbb E) \ar[r] \ar[d] 
&\otimes^k(\mathfrak g/\mathfrak h)\otimes \mathbb E^* \ar[r] \ar[d]
& 0\\
0 \ar[r] 
& V_{k-1}(\mathbb E) \ar[r]
& V_k(\mathbb E) \ar[r]  
& S^k(\mathfrak g/\mathfrak h)\otimes \mathbb E^* \ar[r]
& 0
}
\end{equation}
where the most right vertical arrow is given by the symmetrization.
 
Next, notice the Frobenius reciprocity \ref{thm.1.6} works without any
change in the proof. 

\begin{prop}[Frobenius reciprocity, \cite{ES}]\label{prop.3.7}
For all finite dimensional representations $\mathbb E$ and $\mathbb F$ of $H$, 
there are the canonical isomorphisms 
$$
\operatorname{Hom}_H(\mathbb F^*,\bar V(\mathbb E)) =
\operatorname{Hom}_{(\bar{\mathfrak U}(\mathfrak g),H)}(\bar V(\mathbb
F),\bar V(\mathbb E)).
$$
\end{prop}

\begin{proof}
If we are given a homomorphism $\Ph\in 
\operatorname{Hom}_{(\bar{\mathfrak U}(\mathfrak g),H)}(\bar V(\mathbb
F),\bar V(\mathbb E))$,
we define $\ph:\mathbb F^*\to \bar V(\mathbb E)$ by restriction.

On the other hand, having a $\ph\in 
\operatorname{Hom}_H(\mathbb F^*,\bar V(\mathbb E))$, we define for all
$x\in\bar{\mathfrak U}(\mathfrak g)$ and $v\in \mathbb F^*$,
$$
\Phi(x\otimes v) = x \otimes_{\bar{\mathfrak U}(\mathfrak h)} \ph(v)
,$$
which extends linearly, if well defined. This is again checked by noticing 
that for all
$X\in\mathfrak h$ and $v\in\mathbb F^*$,
$$
\Ph(X\otimes v - 1\otimes X\cdot v) = X\otimes \ph(v) - 1\otimes \ph(X\cdot
v) = X\otimes \ph(v) - 1\otimes X\cdot\ph(v)
,$$ which completes the proof.
\end{proof}

\subsection{Lifting homomorphisms}\label{3.8} 
Let us restrict ourselves to the parabolic geometries, i.e., semisimple Lie
groups $G$ with parabolic subalgebras $P$, the generalized Verma modules, and
their semiholonomic versions.

While the homomorphisms between the generalized Verma
modules,
$$\operatorname{Hom}_{(\mathfrak U(\mathfrak
g),H)}(V(\mathbb F),V(\mathbb E))
$$ 
are often very well understood, very little is known about the spaces 
$$
\operatorname{Hom}_{(\bar{\mathfrak U}(\mathfrak
g),H)}(\bar V(\mathbb F),\bar V(\mathbb E))
$$ 
which we are interested in, now. 

The strategy proposed in \cite{ES} is to discuss the possible liftings of
the existing homomorphisms $V(\mathbb F)\to V(\mathbb E)$ to morphisms 
$\bar V(\mathbb F)\to \bar V(\mathbb E)$ with respect to the canonical
projection. Moreover, due to the Frobenius reciprocity, this is equivalent
to the search for the dashed $H$-module morphisms in the following commutative diagram:
\begin{equation}\label{liftingdiagram}
\xymatrix@R=5mm@C=20mm{
& \bar V(\mathbb E) \ar[d]
\\
\mathbb F^* \ar@{-->}[ur] \ar[r] 
& V(\mathbb E) 
}
\end{equation}
In turn, for irreducible modules $\mathbb E$, $\mathbb F$, 
this is equivalent to finding a highest weight vector in $\bar
V(\mathbb E)$ covering the relevant highest weight vector in $V(\mathbb
E)$.

Next, recall from \ref{2.4} that the order of homomorphism $\Ph:V(\mathbb F)\to
V(\mathbb E)$ is the lowest
$k$ such that $\Ph$ maps $\mathbb F^*$ into $V_k(\mathbb E)$. The order of
homomorphisms $\Ph$ between the semiholonomic Verma modules is defined in the same
way.

Then, again following the Klein's model case, the
symbol of $\Ph$ is
$$
\si(\Ph):\mathbb F^* \to \bar V_k(\mathbb E)\to \bar V_k(\mathbb E)/\bar
V_{k-1}(\mathbb E) = \otimes^k(\mathfrak g_-)\otimes\mathbb E^*
.$$ 

\begin{prop}[\cite{ES}]
A homomorphism $V(\mathbb F)\to V(\mathbb E)$ of order at most two 
always lifts to a homomorphism $\bar V(\mathbb F)\to \bar V(\mathbb E)$.
\end{prop}
\begin{proof}
The claim is equivalent to the existence of an $H$-equivariant splitting of
the canonical projection $\bar V_2(\mathbb E)\to V_2(\mathbb E)$. Following
\cite{ES}, we define such a splitting by identity on $V_1\mathbb E=\bar
V_1\mathbb E$, and we determine it completely by mapping
$$
V_2(\mathbb E) \ni XYe \mapsto \frac12(XY + YX + [X,Y])e\in \bar V_2(\mathbb
E)
$$
for all $e\in \mathbb E^*$, $X,Y\in\mathfrak g$. Checking its equivariance
is straightforward.
\end{proof}

This simple proposition implies again that all first and
second order linear invariant operators extend canonically from the
Klein's models to the entire category of the corresponding Cartan
geometries. 


Let us now restrict to $|1|$-graded parabolic geometries, which is the case
with all our examples. Then our
definition of symbol is compatible with the homogeneities in terms of the
actions of the grading elements $E$ in $\mathfrak g_0$. Here we also enjoy the
following proposition. In general, we could also think about the finer filtering
of the induced modules governed by the action of $E$, as we are doing in the
filtering of the $\mathfrak g$-modules when viewed as $\mathfrak p$-modules. 

\begin{prop}[\cite{ES}]
For all $|1|$-graded parabolic geometries and 
irreducible $P$-modules $\mathbb E$, $\mathbb F$, the homomorphisms $V(\mathbb
F)\to V(\mathbb E)$, or $\bar V(\mathbb F)\to \bar V(\mathbb E)$, are
determined by their symbols.

If the homomorphism $\bar \Phi$ of the semiholonomic Verma modules covers
$\Phi$, then the symbol of $\Phi$ 
is obtained by symmetrization of the symbol of $\bar
\Phi$.

The order of the homomorphism is given as the difference of the actions of
the grading element $E\in\mathfrak g_0$ on $\mathbb F$ and $\mathbb E$.
\end{prop}

\begin{proof}
The existence of the grading element in $\mathfrak g_0$ defining the
grading of $\mathfrak g$ shows that the highest weight vector determining a $k$th order
operator must sit in $\otimes^k(\mathfrak g/\mathfrak p)\otimes \mathbb
E^*$. Since $\mathbb F^*$ generates both $V(\mathbb F)$ and $\bar V(\mathbb
F)$, the homomorphisms are uniquely determined by the embeddings $\mathbb F^*$.   

The next claim is obvious from the commutative diagram \eqref{jetdiagram}. 

The final observation is clear since we deal with $|1|$-graded geometries,
so the degree $k$ of $S^k(\mathfrak g_-)\otimes\mathbb E$ must be just the
mentioned difference.
\end{proof}

\subsection{Curved translation principle}\label{curvedJZ}
We are going to show, that the translation principal extends to some extent
from the homogeneous case to the general curved Cartan geometries. The main
idea (introduced in \cite{ES}) is to show, that many translations of
morphisms $\Ph:V(\mathbb F)\to V(\mathbb E)$ which can be covered by
$\bar\Ph:\bar V(\mathbb F)\to \bar V(\mathbb E)$ lead to results which again
can be covered.  

A special case of the invariant operators are the so called splitting
operators used in \eqref{translation}. There, the embeddings $V(\mathbb
F')\to V(\mathbb F\otimes\mathbb W)$ and projections $V(\mathbb
E\otimes\mathbb W)\to V(\mathbb E')$ are morphisms, whose orders are given
by the relevant position of the dashed modules in the filtration of $\mathbb
W$. Thus, if the difference from the top or bottom, respectively, 
is at most two, we can be sure that the necessary splitting
will exist in the semiholonomic version as well. 

In the semiholonomic case, we either can assume that the filtration of
$\mathbb W$ is of length at most two, or we can restrict ourselves only to
submodules which are at most two steps from the highest component in the
filtration for the embeddings, and at most two steps from the bottom for the
projections.

Then, we can cover the translation from the Klein's model in the curved case
as seen in the next diagram:
$$
\xymatrix@R=5mm@C=6mm{
\bar V(\mathbb F') \ar[r] \ar[d] \ar@(ur,ul)@{-->}[rrrrr]^{\bar D'}
&\bar V(\mathbb F\otimes \mathbb W) \ar@{=}[r] \ar[d]
&\bar V(\mathbb F)\otimes\mathbb W^* \ar[d] \ar[r]^{\bar D\otimes 1}
&\bar V(\mathbb E)\otimes \mathbb W^* \ar@{=}[r] \ar[d]
&\bar V(\mathbb E\otimes \mathbb W) \ar[r] \ar[d]
&\bar V(\mathbb E') \ar[d]
\\
V(\mathbb F') \ar[r] \ar@(dr,dl)@{-->}[rrrrr]_{D'}
&V(\mathbb F\otimes \mathbb W) \ar[r]
&V(\mathbb F)\otimes\mathbb W^* \ar[r]^{D\otimes 1}
&V(\mathbb E)\otimes \mathbb W^* \ar[r] 
&V(\mathbb E\otimes \mathbb W) \ar[r] 
&V(\mathbb E')
}
$$

Summarizing, all homomorphism of Verma modules which 
can be obtained from 1st or 2nd order
morphisms by means of translation described in propositions \ref{transprop1}
and \ref{transprop2}, using only splitting operators of order at most two,
admit the covering by homomorphisms of semiholonomic Verma modules.
Moreover, the symbols of the covered homomorphisms are obtained by
symmetrizations of symbols of those covering ones.  

\subsection{Non-existence}\label{non-existence}
The (non)existence of homomorphisms of semi-holonomic Verma modules can be
sometimes seen directly from the semiholonomic jet module picture, as
discussed in \ref{3.4}. 

In general, the action of $\mathfrak g_1$ on $\bar J^k\mathbb E$ is
horrible, but the restriction of this action to $\otimes^{k-1}\mathfrak
g_-^*\otimes\mathbb E\subset \bar J^k\mathbb E$ is
relatively simple. For all $Z\in \mathfrak g_1$, $\ph\in 
\otimes^{k-1}\mathfrak g_-^*\otimes \mathbb E \subset \bar J^k\mathbb E$, and $X_1,\dots,X_k\in
\mathfrak g_{-1}$ we obtain
\begin{equation}\label{the_action}
\begin{aligned}
(Z\cdot \ph)(X_1,\dots,X_k) &= \sum_{i=1}^{k}\biggl([X_i,Z]\cdot
\ph(X_1,\dots^\wedge,X_k) 
\\
&\qquad - \sum_{j=1}^{i-1} \ph(X_1,
\dots,[[X_i,Z],X_j],\dots^\wedge,X_k) 
\biggr)
\end{aligned}
\end{equation}
where the wedges indicate the relevant omission of arguments.

Now, if there is a $\mathfrak p$-homomorphism $\Ph:\bar J^k\mathbb E\to
\mathbb F$, and both $\mathbb E$ and $\mathbb F$ are irreducible, then
$\Ph$ must vanish on the image of the $\mathfrak p_+$ action. This simple
observation led to complete description of all first order operators in
\cite{SlSo}, for all parabolic geometries. 

The first step was very simple there: fixing the action of the 
semisimple part of $\mathfrak g_0$, and leaving the so called weight (i.e.,
the action of the center $\mathfrak z\subset \mathfrak g_0$) free, the above
condition restricted to the action $\mathfrak g_1$ on
$\mathbb E\subset \bar J^1\mathbb E$ reveals that only the derivatives in
the direction of the smallest distributions corresponding to $\mathfrak g_{-1}$
are feasible and it provides one linear constraint on 
the weight. Then
we can show that the operators exist for all such weights. 

In the one graded
case, this recovers the earlier known fact from conformal Riemannian geometry, 
that fixing the action of the
semisimple part of $\mathfrak g_0$ for $\mathbb E$ and $\mathbb F$, then
each invariant projection of $\mathfrak g_1\otimes \mathbb E$ to $\mathbb F$
yields an invariant operator for a unique weight of $\mathbb E$, cf.
\cite{Feg}.

The same approach gets very much more complicated for higher orders. 
We shall illustrate the procedures on
our simplest $(2|2)$-Grassmannian example, which will conclude our survey.

\subsection{4-dimensional conformal geometries} 
As exploited in \cite{ES}, 
all fundamental representations $\mathbb W$ in the conformal Riemannian
geometry, i.e., for the algebra $\mathfrak s\mathfrak o(n+1,1)$, are of
length at most two and so we obtain no restriction when using them in order
to move from the trivial representation to any other one. 

A particular case is our $(2,2)$-Grassmannian example. Dealing with the 
de~Rham pattern for the regular infinitesimal characters there
(see \eqref{2|2} and notice it contains only operators of order one or their
nontrivial compositions, except the
mysterious fourth order Paneitz operator not depicted there), we immediately
see, that the pattern remains the same for all regular characters. 
In particular, this shows that our description of all Verma module
homomorphisms for the
regular infinitesimal characters in the paragraph after \eqref{2|2} extend
to the curved 4-dimensional conformal geometries, except the Paneitz
operator. Let us discuss this closer now.

We shall use the usual Penrose abstract index
notation, as started in subsection \ref{index_notation}. Thus let write
$X=X^{A'}_A$, $Y=Y^{A'}_A\in \mathfrak g_{-1}$, $Z=Z^A_{A'}\in
\mathfrak g_1$, and consider the trivial
representation with weight $w$. 

We shall first recover the second order Yamabe operator from
\eqref{2|2-singular}. Then we need to compute the action of $Z$ on
$\ph=\ph^A_{A'}\in \mathfrak g_{-1}^*[w]$. In this case, \eqref{the_action}
becomes:
$$
(Z\cdot\ph)(X,Y) = [X,Z]\cdot \ph(Y) +[Y,Z]\cdot\ph(X) - \ph([X,Z],Y)
$$
and, since we deal with trivial representation of the semisimple
part of $\mathfrak g_0$, 
$$[X,Z]\cdot \ph(Y) = wZ(X)\ph(Y).
$$ 
Indeed, the
grading element $E$ acts on $X_A$ by 1/2, thus it acts on $X_{[A,B]}$ by 1, 
as anticipated, while the central part, i.e., the coefficient at the grading
element is obtained by the evaluation. 

Writing down the action of $\mathfrak g_1$ on the densities with
weight $w$ by means of the abstract indices, we arrive at
$\ph([[-,Z],-])^{AB}_{A'B'}=\ph^B_{A'}Z^{A}_{B'}+\ph^A_{B'}Z^B_{A'}$, and so
the action gets the following shape:
\begin{equation}\label{action_Yamabe}
Z^A_{A'}\cdot u^B_{B'} =
wZ^A_{A'}u^B_{B'}+wZ^B_{B'}u^A_{A'}-u^B_{A'}Z^A_{B'}-u^A_{B'}Z^B_{A'}
.\end{equation}
The potential second order operator valued again in densities is obtained
by antisymmetrization in both upper and lower indices. The condition that
the morphism must vanish on the entire image of the action says
$$
2(w+1)Z^{[A}_{[A'}u^{B]}_{B']} = 0
$$
and thus we obtain the value $w=-1$ as the right weight for the densities.
This homomorphism yields the famous Yamabe operator, the conformally
invariant version of the Laplace operator. Such natural operators can be
expressed by a universal formula in terms of any of the metrics in the
conformal class, and notice that even in the case of
the flat conformal sphere $S^4$, the Yamabe operator involves additional
lower order correction term to the Laplacian in the symbol. The reader can
find a detailed explanation of such phenomena for all parabolic geometries
in \cite{CS}. 

\subsection{The Paneitz operator} 
In principle, we should be able to continue along the same line of
arguments, look at image of the action of $\mathfrak g_1$ on $\otimes^3\mathfrak
g_1$ inside $\bar J^4(\mathbb E)$ for the trivial module $\mathbb E=\mathbb R$. 
But such an endeavor gets pretty complicated.  

Thus, it is time to switch to the dual picture.  In the holonomic Verma modules,
we are looking for the singular vectors with trivial weight, in the module
induced by the top-rank exterior forms $\mathbb F$, and we seek for them in
the top layer of $V_4(\mathbb F)$. Recall that we identify $\mathfrak g_{-1}$
with the matrices with 2 rows and columns, and let us write them as
$(y_{11},y_{12},y_{21},y_{22})$. 

A direct check reveals that the relevant singular vector for the Yamabe
operator is the determinant $y_{11}y_{22}-y_{12}y_{21}$, understood as an
element in $S^2(\mathfrak g_{-1})\otimes \mathbb R$ (with the right density
weight).  In our fourth order case, the only singular vector of the right
trivial weight is the square of the determinant (up to constant multiple, of
course).  The computations are tedious but straightforward, and they can be
nicely done, e.g., using Maple.

The situation gets much more complicated in the curved Cartan's worlds. We
want to cover the known singular vector. Since the $y_{ij}$'s do not commute
any more, we have to modify the formula for the determinant appropriately.
As with the jets, we simply consider complete antisymmetrizations of both
upper and lower indices. This gives us more terms than in the symmetric
case:
\begin{equation}\label{non_com_det}  
\frac12(y_{11}y_{22}-y_{21}y_{12} - y_{12}y_{21} + y_{22}y_{11})
\end{equation}
and taking the second power, we obtain 16 terms (instead of 4 in the
holonomic case). 

Now, the crucial observation is that there are actually three independent
options how to perform the two antisymmetrizations over 4 indices - we
choose the first couple and then continue with the remaining one. We
have to consider their linear combinations with the sum of coefficients
equal to one, in order to cover the symmetric singular vector. These are the
only $\mathfrak g_0$-highest weight vectors which project onto the fixed
symmetric singular vector.

Next, as a matter of fact (again quite straightforwardly computed, e.g., in
Maple) the action of the generator $z_{21}$ of $\mathfrak g_1$ is nontrivial
and identical
on all three of the options for the second power of determinant. Thus
this action will never vanish on any of our coverings of the symmetric
singular vector. 

This shows that there cannot be any semi-holonomic Verma module 
homomorphism providing the Paneitz operator. 

\subsection{Final remarks} Analogous result was proved for all the longest
operators in the de Rham pattern in conformal geometries of all even
dimensions $n\ge4$ in \cite{ES}. The semi-holonomic Verma module technique
was first developed there.  

Notice that actually there are the very
exceptional invariant operators extending the flat ones in the de Rham even
in the curved case. This can happen due to the fact that the image of the
universal semi-holonomic jet operator is an algebraic subvariety in the
total jet module, thanks to the Bianchi and Ricci identities for Cartan
connections and their differential consequences.

E.g., in the case of the original Paneitz operator, we simply use the same
formula as in the flat case and the restriction of the corresponding
$\mathfrak g_0$-module homomorphism to the image of the universal jet operator
happens to be equivariant. 

All the other operators in the $(2,2)$-Grassmannian pattern, singular or
regular can be obtained from the most simple ones: first order Dirac and
first order exterior differential. At the same time, all the fundamental
$G$-representations are coming with filtrations of length two or one, so we
can directly use Jantzen-Zuckermann procedure to get all patterns also in
the curved case (except the longest arrows). In higher conformal dimensions, the
Propositions \ref{transprop1}, \ref{transprop2} were necessary to built all
the other long arrows between the forms in the de Rham, see \cite{ES}.

The fourth order
generalized Paneitz operators for the quaternionic-like geometries 
are studied in detail in \cite{Nav,Nav24}. Again, in the de Rham, there are
still invariant operators extending the flat case, but they need a much more
careful approach, cf. \cite{Nav, Nav24}.

The case of the $(3,3)$-Grassmannians is discussed in great detail 
in the parallel paper by
the same authors, \cite{SlSo-prep}.


\begin{thebibliography}{XX}

\bibitem{BEG}     
T.N.\ Bailey, M.G.\ Eastwood, A.R.\ Gover,
Thomas's structure bundle for conformal, projective and related structures,
Rocky Mountain J. {\bf 24} (1994), 1191--1217.

\bibitem{Ba1} R.J.Baston, Verma modules and differential conformal
invariants, J. Differential Geometry, 32 (1990), 851-898.

\bibitem{BC1} B.D. Boe, D.H. Collingwood, A comparison theory for the 
structure of induced representations I., J. of Algebra, 94 (1985), 511-545.

\bibitem{BC2} D.D. Boe, D.H. Collingwood, A comparison theory for the 
structure of induced representations II., Math. Z., 190 (1985), 1-11.

\bibitem{BJ}
W. Borho, J.C. Jantzen, \"Uber primitive Ideale in der Einh\"ullenden 
einer halbeinfachen Lie-Algebra.(German. English summary),
Invent. Math., 39 (1977), 1-53.

\bibitem{CD} D.M.J.\ Calderbank, T.\ Diemer, Differential invariants and
curved Bernstein-Gelfand-Gelfand sequences, J.\ Reine Angew.\ Math.\ \textbf{537}
(2001), 67--103.

\bibitem{CG} A.\ \v Cap, A.R.\ Gover, Tractor calculi for parabolic
geometries, Trans.\ Amer.\ Math.\ Soc.\ \textbf{354} (2002), 1511--1548.

\bibitem{parabook} A.\ \v Cap and J.\ Slov\'ak, Parabolic geometries.\ I, Background
and general theory, Mathematical Surveys and Monographs, vol.\ 154, Amer.\ Math.\ Soc., 
Providence, RI, 2009, x+628 pp. 

\bibitem{CS} A. \v Cap, J. Slov\'ak, 
Bundles of Weyl structures and invariant calculus for parabolic geometries,
Contemp. Math., 788, 
American Mathematical Society, 2023, 53–72.

%
%

%
%
\bibitem{CSS4} A.\ \v Cap, J.\ Slov\'ak, V.\ Sou\v cek, 
Bernstein--Gelfand--Gelfand sequences, Ann.\ of Math.\ \textbf{154} (2001), 97--113. 


\bibitem{Cartan-conf} E. Cartan, Les espaces \`a connexion conforme, Ann.\
Soc.\ Pol.\ Math.\ {\bf 2} (1923), 171--202.

\bibitem{E-notices} M.G. Eastwood, Variations on the de Rham complex,
Notices Amer. Math. Soc. 46 (1999), no. 11, 1368-1376.

\bibitem{ER} M.G. Eastwood, J.W. Rice, Conformally invariant 
differential operators on Minkowski space and their curved analogues, 
Commun. Math. Phys., 109 (1987), 207-228.

\bibitem{ERe} M.G. Eastwood, J.W. Rice,
Erratum: ``Conformally invariant differential operators on
Minkowski space and their curved analogues''
Comm. Math. Phys. 144 (1992), no. 1, 213.

\bibitem{ES} M.G.\ Eastwood, J.\ Slov\'ak, Semi-holonomic Verma modules,
J.\ of Algebra {\bf 197} (1997), 424--448. 

\bibitem{EnSh} T. Enright, B. Shelton, 
Categories of highest weight modules: Applications to
classical Hermitian symmetric
pairs, Memoires AMS, 367, 1987.

\bibitem{Feg} H. Fegan, Conformally invariant first order differential
operators, Quat. J.
Math. Oxford (2), 27 (1976), 371-378.

\bibitem{GJMS} C.R., Graham, R. Jenne, J. Mason, A.J. Sparling, 
Conformally invariant powers of the Laplacian, I: Existence, 
J. London Math. Soc., 46 (1992), 557-565.

%
%
\bibitem{Ki} A.A. Kirillov, Invariant operators over geometric quantities,
In: Current problems in mathematics, Vol 16 (1980), (Russian), VINITI, 3-29.

\bibitem{KMS} I.\ Kol\'a\v r, P.\ Michor, J.\ Slov\'ak, Natural operations in
differential geometry, Springer, 1993, 434 pp.

\bibitem{Ko} B. Kostant, 
Verma modules and the existence of quasi invariant differential operators,
In: Non-Commutative Harmonic Analysis, Springer, LNM 466 (1975), 101–128.

%
\bibitem{Lep} J. Lepowsky, A generalization of the Bernstein-Gelfand-Gelfand
resolution, J. Algebra, 49 (1977), 496-511.

\bibitem{Nav} A. N\'avrat, Non-standard Operators in Almost
Grassmannian Geometry, PhD Dissertation, University of Vienna, 2012, 120pp. 

\bibitem{Nav24} A. N\'avrat, A Grassmannian Analogue of Paneitz Operator, to
appear.

\bibitem{Saw} J. Sawon, Homomorphisms of semiholonomic Verma modules: an exceptional
case, Acta Math. Univ. Comenian., 68 (1999), no. 2, 257-269.

\bibitem{Sha} R.W.\ Sharpe, Differential Geometry, Graduate
Texts in Mathematics 166, Springer--Verlag 1997.

\bibitem{SlSo} J. Slov\'ak, V. Sou\v cek, First order invariant
differential operators for parabolic geometries. In Seminaires \& Congres.
France: French Math. Soc., 2000. p. 249-273.

\bibitem{SlSo-prep} J. Slov\'ak, V. Sou\v cek, Curved translation principle
in generalized conformal calculus, preprint to appear.

\bibitem{SlSu} J. Slov\'ak, R. Such\'anek,
Notes on tractor calculi,
Tutor. Sch. Workshops Math. Sci.,
Birkh\"auser/Springer, Cham, 2021, 31-72.

\bibitem{vogan} D. A. Vogan, Representation of Real Reductive Lie Groups, 
Progress in Mathematics, Vol. 15, Birkhauser, Boston, Cambridge, MA, 1981.

\bibitem{Tan} N.\ Tanaka,  On the equivalence problem associated 
with simple graded Lie algebras, Hokkaido Math.\ J.\ {\bf 8} (1979), 23--84.

\bibitem{Zierau} R. Zierau, Representations in Dolbeault Cohomology, in Representation Theory of Lie Groups
edited by Jeffrey Adams, David Vogan, IAS/Park City Mathematics Series,
American Math. Soc., vol. 8 (2015), 89-146.

\bibitem{Zu} G. Zuckerman, Tensor products of finite and infinite 
dimensional representations of semisimple Lie groups, Ann. Math., 106
(1977), 295-308.

%
\end{thebibliography}
\end{document}